\documentclass[12pt]{amsart}
\usepackage{graphicx, enumerate, amsmath, amssymb,amsthm,manfnt,hyperref,amscd,braket}
\usepackage[utf8]{inputenc} 
\allowdisplaybreaks

\newcommand{\e}{\varepsilon}

\newcommand{\R}{{\mathbb{R}}}
\newcommand{\C}{{\mathbb{C}}}

\newcommand{\T}{\mathbb{T}}
\newcommand{\Z}{\mathbb{Z}}

\newcommand{\tmop}[1]{\ensuremath{\operatorname{#1}}}
\renewcommand{\Re}{\tmop{Re}}
\renewcommand{\Im}{\tmop{Im}}

\theoremstyle{plain}
\newtheorem{theorem}[equation]{Theorem}
\newtheorem{proposition}[equation]{Proposition}
\newtheorem{lemma}[equation]{Lemma}
\newtheorem{corollary}[equation]{Corollary}
\newtheorem{definition}[equation]{Definition}

\newtheorem{example}[equation]{Example}
\usepackage{comment}

\theoremstyle{remark}
\newtheorem{problem}[equation]{\bf Problem}

\newtheorem{remark}[equation]{Remark}
\numberwithin{equation}{section}

\newcommand{\U}{\mathbb{U}}

\newcommand{\B}{\mathbb{B}}
\newcommand{\Sp}{\mathbb{S}}
\newcommand{\D}{\mathbb{D}}

\newcommand{\Sb}{\mathbb{G}}

\title[Szeg\H{o} Projection]{Hardy spaces and Szeg\H{o} projection on quotient domains}
\author{Liwei Chen}
\address[Liwei Chen]{Syracuse University, Department of Mathematics, Syracuse, NY 13244}
\email{lchen125@syr.edu}

\author{Yuan Yuan}
\address[Yuan Yuan]{Syracuse University, Department of Mathematics, Syracuse, NY 13244}
\email{yyuan05@syr.edu}

\subjclass[2010]{Primary: 32A25, Secondary: 32A35, 32A07,42B20, 42B30}
\keywords{Szeg\H{o} projection, $H^p$-spaces, holomorphic covering, Thullen domains}
\begin{document}

\maketitle
\begin{abstract}
The Hardy spaces are defined on the quotient domain of a bounded complete Reinhardt domain by a finite subgroup of $U(n)$. The Szeg\H{o} projection on the quotient domain can be studied by lifting to the covering space. This setting builds on the solution of a boundary value problem for holomorphic functions. In particular, when the covering space is either the polydisc or the unit ball in $\C^n$, the boundary value problem can be solved. Applying this theory in $\C^2$, we further obtain sharp results on the $L^p$ regularity of the Szeg\H{o} projection on the symmetrized bidisc, generalized Thullen domains, and the minimal ball.
\end{abstract}

\section{Introduction}

The study of the Hardy spaces on the unit disc $\D$ is a classical subject dated back at least to Riesz \cite{R23}. The higher  dimensional  generalization on various domains, such as 
polydisc $\D^n$,  unit ball $\B^n$,  Siegel's half space $\U^n$, real Euclidean space $\R^n$, Heisenberg group $\mathbb{H}^n$, and even general bounded domains in $\C^n$ with $C^2$ smooth boundary, is also a natural question which attracts substantial attentions (see \cite{KV72,R69,R80,FS72,Ste93} and the references therein). Roughly speaking, the Hardy space $H^p$ on a domain can be identified with the closed subspace of the $L^p$-space on the boundary, consisting of the boundary limit of holomorphic functions on the domain. In particular, the orthogonal projection $S:L^2\to H^2$ is defined to be the Szeg\H{o} projection and it can be recognized as a singular integral on the boundary. Problems on the regularity of the Szeg\H{o} projection on different function spaces are of immense interests in complex analysis, harmonic analysis, and operator theory. On various bounded smooth pseudoconvex domains, the Sobolev regularity of the Szeg\H{o} projection is well-known (\cite{PS77,NRSW89,MS97,CD06}). One important progress in the direction regarding domains of less smooth boundary was recently made by Lanzani and Stein. They developed the theory of Hardy spaces with arbitrary positive measure with respect to the Lebesgue measure on bounded strongly pseudoconvex domains with minimal $C^2$ smooth boundary in \cite{LS16}, and proved the sharp $L^p$ regularity of the Szeg\H{o} projection in \cite{LS17}. It becomes more intriguing when the domain is not smooth and there have been quite some interesting works recently on the Hardy spaces and the Szeg\H{o} projection on non-smooth domains (\cite{LS04, PS08, MSZ13, MP17, GGLV21, HT22, GG23}).

\medskip

Among recent studies on non-smooth domains in $\C^n$, those with quotient structures attracts substantial interests. Let $X, \Omega$ be bounded domains  in $\mathbb{C}^n$ such that there exists a proper holomorphic   map $\Phi:X\to\Omega$. An application of the classical
Remmert proper mapping theorem shows that $\Phi$ is a branched cover of finite order (see for example \cite{Bel82}). More precisely, there exists an analytic subvariety $Z$ in $X$ such that $\Phi: X\setminus Z \rightarrow \Omega\setminus \Phi(Z)$ is a  covering with the finite group of deck transformations $G$. We simply call $\Omega$ the quotient domain of $X$. 
In particular, when $X$ is $G$-invariant, $\Omega$ is homeomorphic to $X/G$. 
The systematic study on quotient domains was firstly carried out for the Bergman projection in \cite{CKY20}. The motivating example in that article is the symmetrized polydisc $\mathbb{G}_n$, a quotient domain of the polydisc $\D^n$ by the permutation group $S_n$ of $n$ elements. While 
\cite{CKY20} only considers the quotient domains of the polydisc, the quotient domains of the unit ball $\B^n$ in $\C^n$ can be handled in the same manner. The readers may refer to \cite{BCEM22, DM23} for more study of the Bergman projection on quotient domains. 
A natural question is whether this setting can be carried over to study the Szeg\H{o} projection on the quotient domains. There are some recent studies: the authors in \cite{MSZ13} defined a notion of the Hardy spaces on the symmetrized polydisc $\mathbb{G}_n$ and then the $L^p$ regularity of the corresponding Szeg\H{o} projection was studied in \cite{HT22}; a notion of the Hardy spaces on the quotient domains of $\D^n$ was defined similarly to that in \cite{MSZ13} and the corresponding Szeg\H{o} projection was studied in \cite{GG23}. However, the measure on the boundary of the covering space $X$ does not seem to be compatible with the pullback of the natural surface measure in the setting of these recent studies. 

\medskip

In this article, the set up of the Hardy spaces is based on more appropriate surface measures. Let $(X, \Omega, \Phi, G)$ be the quadruple associated to the quotient domain $\Omega$ of $X$ as defined above. We further assume that $\Phi$ extends smoothly to the topological boundary $\partial X$ of $X$. Let $\Gamma_X$ be a submanifold of $\partial X$ such that $\Gamma_X$ is invariant under $G$. Denote $\Phi (\Gamma_X)$ by $\Gamma_\Omega$. The induced Lebesgue measures on $\Gamma_X, \Gamma_\Omega$ are denoted by $d\sigma, d\mu$ respectively. Write the pullback surface measure $\Phi^* \mu = w \sigma$, where $w$ is a nonnegative density function. 
Let $x=(x_1, \cdots, x_d)$ be the real parameters on $\Gamma_X$ and define 
\[
J_{\R}Z=\left(\begin{array}{cccc} \frac{\partial z_1}{\partial x_1} & \frac{\partial z_2}{\partial x_1} & \dots & \frac{\partial z_n}{\partial x_1} \\ \dots \\ \frac{\partial z_1}{\partial x_d} & \frac{\partial z_2}{\partial x_d} & \dots & \frac{\partial z_n}{\partial x_d} \end{array}\right)_{d\times n}.
\]
Let $J_{\C}\Phi$ be the complex Jacobian matrix of the covering map $\Phi$ and define
\begin{equation}\notag
A=J_{\R}Z\cdot J_{\C}\Phi.
\end{equation}
We show in Proposition \ref{den} that $$\Phi^*\mu = \sqrt{\det[\Re(A\overline{A^t})]} dx_1 \cdots dx_d.$$
This is the key difference in our set up of the surface measure from that of \cite{MSZ13, GG23}, where the authors directly use the volume measure $w = |\det J_{\C}\Phi|^2$. Note that  a simply connected planar domain $\Omega$ with a conformal map $\varphi: \D \rightarrow \Omega$ can be considered as a special case of a quotient domain in one dimension.   
In this classical setting,  the pullback of the arc length of $\partial\Omega$ is given by $\varphi^*\mu=|\varphi'|d\theta$ (see \cite{LS04} and the references therein). 
In this sense, our set up is a natural generalization of the one dimensional case. 

Two different types of domain $X$ can be considered here:  $X$ is either a bounded complete Reinhardt domain of $C^2$ boundary with $\Gamma_X=\partial X$ or the polydisc $\D^n$ with $\Gamma_X=\T^n$. To define the Hardy spaces on $\Omega$, the following boundary value problem for holomorphic functions on $X$ is considered.

\begin{problem}\label{CY}
Let $w$ be a nonnegative $G$-invariant function (see Definition \ref{funinvariant}) on $\Gamma_X$. Find the necessary and sufficient condition on $w$, such that there exists a non-zero $G$-invariant holomorphic function $g$ on $X$ with the non-tangential limit $g^*$ on $\Gamma_X$ satisfying $|g^*|^2=w$ almost everywhere on $\Gamma_X$.
\end{problem}

Note that on a simply connected planar domain $\Omega$ with a conformal map $\varphi: \D \rightarrow \Omega$ and $w=|\varphi'|$, one can simply take $g = (\varphi')^\frac{1}{2}$ on $\D$ as in \cite{LS04}. 
More generally, the Herglotz representation: 
\begin{equation}\label{Her}\notag
g(re^{i\theta}):=\exp\left[\frac{1}{2\pi}\int_0^{2\pi}\frac{e^{it}+re^{i\theta}}{e^{it}-re^{i\theta}}\log(w^{\frac{1}{2}}(t))\,dt\right]
\end{equation}
provides a solution to this problem on $X=\D$, assuming $\log w \in L^1(\mathbb{T})$. 
When $n >1$, Problem \ref{CY} for general domain $X$ remains open. However, we are able to provide a sufficient condition $\log w \in L^1(\Gamma_X)$ for Problem \ref{CY} on both $X=\D^n$ and $X=\B^n$ (Theorem \ref{normsquare} and Theorem \ref{normsquareb}).
The main ingredient is the Poisson integral of a singular measure on the boundary.  On $\D^n$, it is essentially the classical Fourier series (\cite{R69}); while on $\B^n$, we apply Rudin's variation of Aleksandrov's Modification Theorem (\cite{Ale82,R83}) in the solution of the famous existence problem of the inner function on $\mathbb{B}^n$. Indeed, a similar version of Problem \ref{CY} in the function algebra setting $H^{\infty}(\D^n)$ or $H^{\infty}(\B^n)$ has been studied in the literature (\cite{R69,R80,R83}). This demonstrates a connection between the setting of our Hardy space theory and the classical function algebra.

When the existence of $g$ in Problem \ref{CY} is assumed, $(X, \Omega, \Phi, G)$ is called an admissible quadruple (see Definition \ref{con}). In such a case, the Hardy space $H^p(\Omega)$ is defined to be the set of all holomorphic functions $f$ on $\Omega$ such that  \begin{equation}\notag
\|f\|^p_{p, \Omega}=\frac{1}{|G|} \sup_{0 < r < 1} \int_{\Gamma_X}  |(f\circ \Phi)(rz)|^p |g(rz)|^2 d\sigma(z) < \infty. 
\end{equation}
 By the classical Hardy space theory on $X$, the Hardy space $H^p(\Omega)$ can be identified with a closed subspace of $L^p(\Gamma_\Omega)$, consisting of the non-tangential limit $f^*$. With a suitable choice of $g$, this is compatible with the classical case when the boundary of $\Omega$ is sufficiently smooth.  The Szeg\H{o} projection $S_{\Omega}$ is then defined as the orthogonal projection $S_{\Omega}:L^2(\Gamma_\Omega) \to H^2(\Omega)$. The first main theorem in this article is the following Bell type  transformation formula for the Szeg\H{o} projection (cf. \cite{Bel82}).

\medskip

{\bf Theorem A.}
Let $K_X(z,\zeta)$ be the Szeg\H{o} kernel on $X$ with $(z,\zeta)\in X\times\Gamma_X$, and $f\in L^2(\Gamma_{\Omega})$. Then 
$$S_{\Omega}(f)(z)=\frac{1}{|G|} \Phi_*\left( \int_{\Gamma_X}\frac{K_X(z,\zeta)}{g(z)\overline{g^*(\zeta)}}f(\Phi(\zeta))w(\zeta)d\sigma(\zeta)\right).$$

To study the $L^p$ regularity of the Szeg\H{o} projection $S_{\Omega}$, we derive the relation between $S_{\Omega}$ and the Szeg\H{o} projection $S_{X}$ on $X$. One notes that while the Hardy space $H^p(\Omega)$ depends  on the choice of $g$, the $L^p$ regularity of the Szeg\H{o} projection actually does not.

\medskip 

{\bf Theorem B.}
For $1<p<\infty$, $S_{\Omega}:L^p(\Gamma_{\Omega})\to H^p(\Omega)$ is bounded if and only if $S_{X}:L^p(\Gamma_X,w^{1-\frac{p}{2}})\to H^p(X,w^{1-\frac{p}{2}})$ is bounded restricted to the subset of $G$-invariant functions.

\medskip

While the framework can be applied to study the $L^p$ regularity of the Szeg\H{o} projection on general quotient domains, to illustrate the idea, we consider some examples in $\C^2$: the symmetrized bidisc
\[
\Sb^2=\left\{(z_1+z_2, z_1z_2)\in \mathbb{C}^2 : (z_1, z_2) \in \D^2)\right\}
\]
and the generalized Thullen domains of type $(m,k)$ for $m,k\in\Z^+$
\[
\Omega_{m,k}=\left\{(\zeta_1,\zeta_2)\in\C^2: |\zeta_1|^{\frac{2}{m}}+|\zeta_2|^{\frac{2}{k}}<1\right\}.
\]
In view of Theorem B, the sufficient condition of the $L^p$ regularity of the Szeg\H{o} projection can be obtained by verifying the Muckenhoupt condition, which belongs to the classical theory of the singular integral operators. Here we apply the explicit formulations in \cite{MZ15, WW21}.  The sufficient conditions turn out to be necessary as well by carefully examining the endpoints of the intervals.

\medskip

{\bf Theorem C.} 
The Szeg\H{o} projection on the symmetrized bidisc $\Sb^2$ is $L^p$ bounded  if and only if $\frac{4}{3}<p<4$.

\medskip

{\bf Theorem D.} 
For $m\ge2$ or $k\ge2$, the Szeg\H{o} projection on $\Omega_{m,k}$ is $L^p$ bounded if and only if $p\in\left(\frac{2m+4}{m+4},\frac{2m+4}{m}\right)\cap\left(\frac{2k+4}{k+4},\frac{2k+4}{k}\right)$.

\medskip

 On the symmetrized bidisc, the $L^p$ regularity of the Szeg\H{o} projection is parallel to that of the Bergman projection obtained in \cite{CJY23, HW23}. However, on the Thullen domain $\Omega_{1,k}$
(\cite{Thu31,HP88}) with $k=2, 3, 4, \cdots$, 
the $L^p$ regularity of the Szeg\H{o} projection differs significantly from that of the Bergman projection. More precisely, the Bergman projection on $\Omega_{1,k}$ is $L^p$ bounded for $1<p<\infty$ and all $k\ge2$ (see \cite[\S 6.1]{HW20} and \cite{Huo18}), while  the Szeg\H{o} projection on $\Omega_{1,k}$ is $L^p$ bounded if and only if $p\in\left(\frac{2k+4}{k+4},\frac{2k+4}{k}\right)$ for $k\ge2$. Note that as $k\to\infty$, $\Omega_{1,k}$ will collapse to a unit disc $\D$ and the dimension degenerates to one. The $L^p$ regularity of the Szeg\H{o} projection detects the change of the geometry while the Bergman projection fails to do so. It is also worth pointing out that in dimension one, B\'{e}koll\'{e} showed that the Szeg\H{o} projection is strictly stronger than the Bergman projection in  the sense of  the $L^p$ regularity when the  domain is not  smooth enough (\cite{Bek86}). The analogue phenomenon in higher dimensions is yet to be discovered. The Thullen domains $\{\Omega_{1,k}\}$ provide a family of supportive examples with Lipschitz continuous boundary, on which the $p$ range  is a non-degenerate proper subset in $(1, \infty)$  for the Szeg\H{o} projection to be $L^p$ regular. 


\medskip

The article is organized as follows. In section \S\ref{Secweight}, we derive the change of variables formula of the induced Lebesgue measure on the boundary in the general setting.  In section \S\ref{generalsetup}, we aim to develop the general theory of the Hardy spaces and the Szeg\H{o} projection. For the admissible quadruple $(X, \Omega, \Phi, G)$ (see Definition \ref{con}), we derive a Bell type transformation formula for the Szeg\H{o} projection and further connect the Szeg\H{o} projection  $S_{\Omega}$  with the Szeg\H{o} projection $S_{X}$.  
 In section \S\ref{SecDn}, we apply the theory to  the polydisc $\D^n$ and obtain the sharp result on the $L^p$-regularity of the Szeg\H{o} projection on the symmetrized bidisc. In section \S\ref{SecBn}, we apply the theory to the ball $\B^n$ and obtain sharp results on the $L^p$ regularity of the Szeg\H{o} projection on generalized Thullen domains and the minimal ball in $\C^2$.


\section{The change of variables formula}\label{Secweight}


Let $\Omega$ be a domain in $\C^n$.  Assume that there exist a domain  $X\subset\C^n$ and a proper holomorphic   map
\[
\Phi:X\to\Omega,
\]
which extends smoothly to the topological boundary $\partial X$. It follows from the classical Remmert proper mapping theorem that $\Phi$ is a branched cover of finite order (\cite{Bel82}).

\begin{definition}
\label{partbdry}
Let $G$ be the finite group acting on $X$ such that $\Omega$ is homeomorphic to $X/G$. Let $\Gamma_X$ be a submanifold of the  topological  boundary $\partial X$ such that $\Gamma_X$ is invariant under $G$. $\Gamma_{\Omega}=\Phi(\Gamma_{X})$ is defined to be the induced boundary on $\Omega$.
\end{definition}

We give two examples which are our main concern in this paper.

\begin{example}
\begin{enumerate}
\item $X=\D^n$ is the polydisc, $\Gamma_{\D^n}=\T^n$ is the distinguished boundary and $G$ is a finite subgroup of $U(n)$.
Then $\Gamma_{\Omega}=\Phi(\T^n)$ is the induced distinguished boundary of $\Omega$.
\item $X=\B^n$ is the unit ball, $\Gamma_{\B^n}=\Sp^{2n-1}$ is  the geometric boundary, and $G$ is a finite subgroup of $U(n)$. Then $\Gamma_{\Omega}=\Phi(\Sp^{2n-1})$ is the induced boundary of $\Omega$.
\end{enumerate}
\end{example}

\begin{definition}\label{weight}
Let $\mu$ be a positive Borel measure which is absolutely continuous with respect to the induced Lebesgue measure on $\Gamma_{\Omega}$. 
Denote the induced Lebesgue measure on $\Gamma_X$ by $d\sigma$ and the cardinality of $G$ by  $|G|$.
If $w_{\mu}$ is a nonnegative function satisfying
\begin{equation}\label{change}
\frac{1}{|G|}\int_{\Phi^{-1}(E)\subset\Gamma_X}w_{\mu}d\sigma=\int_{E\subset\Gamma_{\Omega}}d\mu
\end{equation}
for any measurable set $E$, then $w_\mu$ is called the density of $\Phi^*\mu$,
the pullback of $\mu$, with respect to $\sigma$.
\end{definition}

To study the classical Hardy spaces on $\Omega$, the surface integrals over (the level sets of) the boundary $\Gamma_{\Omega}$ with respect to the Lebesgue measure are considered. We thus restrict ourself to the case of
\[
\mu=m_{\Gamma_{\Omega}}
\]
being the induced Lebesgue measure on $\Gamma_{\Omega}$, and the density function is simply denoted by
\[
w:=w_{m_{\Gamma_{\Omega}}}.
\]
Namely, we write $$\Phi^*m_{\Gamma_{\Omega}} = w \sigma.$$


The starting point to define the Hardy spaces is to calculate the density function of the pullback measure. In order to achieve that, we first state an elementary calculus lemma.

\begin{lemma}\label{cal}
Let $A: \mathbb{R}^n \rightarrow \mathbb{R}^m$ be a $n \times m$ matrix with $m \geq n$. Then the $n$-dimensional volume element of $A(\mathbb{R}^n)$ is 
$\sqrt{\det \left(A \cdot A^t\right)}$. In other words, $m_n(A(I))=\sqrt{\det \left(A \cdot A^t\right)}$, where $I$ is the unit cube in $\R^n$.
\end{lemma}

The change of variables formula in the real case is as follows.

\begin{proposition}
Let $U$ be an oriented smooth manifold and $\dim_{\R}(U)=d$ and $\varphi:U\to\R^k$ be a smooth map, with $V=\varphi(U)\subset\R^k$, where $k\ge d$.
Let $x=(x_1, \cdots, x_d)$ be a local coordinate in $U$ and 
denote the induced Lebesgue measure on $V$ by  $dm_V$. Then 
\begin{equation}
\label{pbjacobian}
\varphi^{*} dm_V=\sqrt{\det[(J_{\R}\varphi)(J_{\R}\varphi)^t]}dx_1\cdots dx_d,
\end{equation}
where $J_{\R}\varphi$ is the $d\times k$ real Jacobian matrix of $\varphi$. 
\end{proposition}

\begin{proof}
For any $p \in U$, by Lemma \ref{cal}, 
\begin{equation*}
\begin{split}
(\varphi^{*}dm_V)(p){\left(\frac{\partial}{\partial x_1}\wedge\cdots\wedge \frac{\partial}{\partial x_d}\right)\vline}_{p} &= dm_V (\varphi(p)) \left( \varphi_{*} {\left(\frac{\partial}{\partial x_1}\wedge\cdots\wedge \frac{\partial}{\partial x_d}\right)\vline}_{p} \right)\\
&= m_d\left((J_{\R}\varphi)(p) (I)  \right) \\
&=\sqrt{\det[(J_{\R}\varphi)(J_{\R}\varphi)^t]}(p),
\end{split}
\end{equation*}
where $I$ is the unit cube in $\R^d$. It follows that $\varphi^{*} dm_V=\sqrt{\det[(J_{\R}\varphi)(J_{\R}\varphi)^t]}dx_1\cdots dx_d$.
\end{proof}

We now apply the above result to  a proper holomorphic   map
\[
\Phi:X\to\Omega,
\]
from a domain $X$ in $\C^n$ to a bounded domain $\Omega$ in $\C^n$, where $\Phi$ extends smoothly to $\partial X$.
 The complex Jacobian matrix of $\Phi=(\phi_1,\dots,\phi_n)$  will be denoted by
\[
J_{\C}\Phi=\left(\begin{array}{cccc} \frac{\partial \phi_1}{\partial z_1} & \frac{\partial \phi_2}{\partial z_1} & \dots & \frac{\partial \phi_n}{\partial z_1} \\ \dots \\ \frac{\partial \phi_1}{\partial z_n} & \frac{\partial \phi_2}{\partial z_n} & \dots & \frac{\partial \phi_n}{\partial z_n} \end{array}\right)_{n\times n}.
\]
Assume $\dim_{\R}(\Gamma_X)=d$ with real parameters $x=(x_1,x_2,\dots,x_d)$. Let
\[
J_{\R}Z=\left(\begin{array}{cccc} \frac{\partial z_1}{\partial x_1} & \frac{\partial z_2}{\partial x_1} & \dots & \frac{\partial z_n}{\partial x_1} \\ \dots \\ \frac{\partial z_1}{\partial x_d} & \frac{\partial z_2}{\partial x_d} & \dots & \frac{\partial z_n}{\partial x_d} \end{array}\right)_{d\times n}
\]
and
\begin{equation}
\label{matrixA}
A:=J_{\R}Z\cdot J_{\C}\Phi.
\end{equation}

The density function can be calculated in the following formula.

\begin{proposition}\label{den}
Denote $\mu=m_{\Gamma_{\Omega}}$ to be the induced Lebesgue measure on $\Gamma_{\Omega}=\Phi(\Gamma_{X})$. Then under the parameterization $x=(x_1,x_2,\dots,x_d)$ of $\Gamma_X$,
$$\Phi^*\mu = \sqrt{\det[\Re(A\overline{A^t})]} dx_1 \cdots dx_d,$$
which is  almost everywhere positive on $\Gamma_X$.
\end{proposition}

\begin{proof}
By identifying  $\Phi$ as a smooth map from $\Gamma_X$ (with real parameters) into $\R^{2n}$, it follows from \eqref{pbjacobian} that $$\Phi^*\mu =\sqrt{\det[J_{\R}\Phi (J_{\R}\Phi)^t]}dx_1 \cdots dx_d .$$
For the map $\Phi=(\phi_1,\dots,\phi_n)$ in $\C^n$, let $\phi_j=u_j+iv_j$ for $j=1,\dots,n$, where $u_j$ and $v_j$ are real-valued functions. Then
\[
J_{\R}\Phi=\left(\begin{array}{ccccc} \frac{\partial u_1}{\partial x_1} & \frac{\partial v_1}{\partial x_1} & \dots & \frac{\partial u_n}{\partial x_1} & \frac{\partial v_n}{\partial x_1} \\ \dots \\ \frac{\partial u_1}{\partial x_d} & \frac{\partial v_1}{\partial x_d} & \dots & \frac{\partial u_n}{\partial x_d} & \frac{\partial v_n}{\partial x_d} \end{array}\right)_{d\times 2n}.
\]
On the other hand, note that $\Phi=(\phi_1,\dots,\phi_n)$ is holomorphic. By the chain rule, one obtains
\begin{equation*}
\begin{split}
A=J_{\R}Z\cdot J_{\C}\Phi
&=\left(\begin{array}{cccc} \frac{\partial \phi_1}{\partial x_1} & \frac{\partial \phi_2}{\partial x_1} & \dots & \frac{\partial \phi_n}{\partial x_1} \\ \dots \\ \frac{\partial \phi_1}{\partial x_d} & \frac{\partial \phi_2}{\partial x_d} & \dots & \frac{\partial \phi_n}{\partial x_d} \end{array}\right)_{d\times n}\\
&=\left(\begin{array}{cccc} \frac{\partial u_1}{\partial x_1}+i\frac{\partial v_1}{\partial x_1} & \frac{\partial u_2}{\partial x_1}+i\frac{\partial v_2}{\partial x_1} & \dots & \frac{\partial u_n}{\partial x_1}+i\frac{\partial v_n}{\partial x_1} \\ \dots \\ \frac{\partial u_1}{\partial x_d}+i\frac{\partial v_1}{\partial x_d} & \frac{\partial u_2}{\partial x_d}+i\frac{\partial v_2}{\partial x_d} & \dots & \frac{\partial u_n}{\partial x_d}+i\frac{\partial v_n}{\partial x_d} \end{array}\right)_{d\times n}.
\end{split}
\end{equation*}
So
\[
\Re(A)=\left(\begin{array}{cccc} \frac{\partial u_1}{\partial x_1} & \frac{\partial u_2}{\partial x_1} & \dots & \frac{\partial u_n}{\partial x_1} \\ \dots \\ \frac{\partial u_1}{\partial x_d} & \frac{\partial u_2}{\partial x_d} & \dots & \frac{\partial u_n}{\partial x_d} \end{array}\right)_{d\times n}
\]
and
\[
\Im(A)=\left(\begin{array}{cccc} \frac{\partial v_1}{\partial x_1} & \frac{\partial v_2}{\partial x_1} & \dots & \frac{\partial v_n}{\partial x_1} \\ \dots \\ \frac{\partial v_1}{\partial x_d} & \frac{\partial v_2}{\partial x_d} & \dots & \frac{\partial v_n}{\partial x_d} \end{array}\right)_{d\times n}.
\]
Therefore,
\[
J_{\R}\Phi=\left[\Re(A), \Im(A)\right]_{d\times 2n}\cdot E_1\cdots E_k,
\]
where $E_1,\dots,E_k$ are elementary $2n\times 2n$ matrices (switching the columns). Note that
\[
E_jE_j^t=I_{2n\times 2n}
\]
for $j=1,\dots,k$. One obtains
\begin{align*}
J_{\R}\Phi(J_{\R}\Phi)^t
&=\left[\Re(A), \Im(A)\right]_{d\times 2n}\cdot E_1\cdots E_kE_k^t\cdots E_1^t\left[\Re(A), \Im(A)\right]_{d\times 2n}^t\\
&=\left[\Re(A), \Im(A)\right]_{d\times 2n}\left[\Re(A), \Im(A)\right]_{d\times 2n}^t\\
&=\Re(A)\Re(A^t)+\Im(A)\Im(A^t).
\end{align*}
Since $A=\Re(A)+i\Im(A)$, a direct computation shows
\begin{align*}
A\overline{A^t}
&=[\Re(A)+i\Im(A)]\cdot[\Re(A^t)-i\Im(A^t)]\\
&=[\Re(A)\Re(A^t)+\Im(A)\Im(A^t)]+i[-\Re(A)\Im(A^t)+\Im(A)\Re(A^t)].
\end{align*}
Therefore, $J_{\R}\Phi(J_{\R}\Phi)^t=\Re(A\overline{A^t})$, and thus
\[
\Phi^*\mu = \sqrt{\det[\Re(A\overline{A^t})]} dx_1 \cdots dx_d.
\]
Moreover, since $\Phi$ is a local differeomorphism from $\Gamma_X$ to $\Gamma_\Omega$ almost everywhere, $J_{\R} \Phi$ is non-singular almost everywhere on $\Gamma_X$. It follows that $\det[\Re(A\overline{A^t})]\neq0$ almost everywhere on $\Gamma_X$.
\end{proof}

\begin{remark}\label{arclength}
When $n=1$, let $X=\D$ be the unit disc in $\C$ and $\Phi=\varphi:\D\to\Omega$ be a conformal map that extends continuously to $\overline{\D}$. Under the parametrization $z=e^{i\theta}$ on the unit circle $\T$, a direct computation shows $\Phi^*m_{\Gamma_{\Omega}}=|\varphi'|d\theta$, which is the pull-back of the arc length of $\Gamma_{\Omega}$ in the classical case (see \cite{LS04}).
\end{remark}

\begin{corollary}\label{leb}
Let $g: X\subset \mathbb{C}^n \rightarrow X\subset \mathbb{C}^n$ be a unitary transformation and $g(\Gamma_X) = \Gamma_X$. Then  $g^*\sigma = \sigma$, where $d\sigma$ is the induced Lebesgue measure on $\Gamma_X$.
\end{corollary}

\begin{proof}
Applying $\Phi=g: X\rightarrow X$ in Proposition \ref{den}, 
\begin{equation*}
\begin{split}
g^*\sigma &=\sqrt{\det[\Re(A\overline{A^t})]} dx_1 \cdots dx_d =  \sqrt{\det[\Re(J_{\R}Z \cdot g \cdot \overline{(J_{\R}Z \cdot g)^t})]} dx_1 \cdots dx_d \\
&=\sqrt{\det[\Re(J_{\R}Z \cdot \overline{(J_{\R}Z) ^t})]} dx_1 \cdots dx_d ,
\end{split}
\end{equation*}
as $g \cdot \overline{g^t} $ is the identity matrix.
 Similarly, if we apply Proposition \ref{den} to the identity map $I: X \rightarrow X$, then $\sigma = I^*\sigma = \sqrt{\det[\Re(J_{\R}Z \cdot \overline{J_{\R}Z^t})]} dx_1 \cdots dx_d$. It thus follows that $g^*\sigma = \sigma$.
\end{proof}


\section{The Hardy spaces and the Szeg\H{o} projection}\label{generalsetup}


Let $X$ be a bounded complete Reinhardt domain in $\C^n$ centered at the origin with $C^2$-smooth boundary and let $\Gamma_X$ be $\partial X$. We also consider the case $X=\D^n$ the polydisc and $\Gamma_X=\T^n$. The classical Hardy spaces on $X$ are defined as follows.

\begin{definition}\label{HponX}
For $0<p<\infty$, the Hardy space $H^p(X)$ is the set of all holomorphic functions $f$ on $X$ such that
\[
\|f\|^p_{p, X}= \sup_{0 <r< 1} \int_{\Gamma_{X}}  |f(rz)|^p d\sigma(z) < \infty.
\]
$\|f\|_{p, X}$ is called the $H^p$-norm of $f$.
\end{definition}

It is well-known that 
for $f\in H^p(X)$, its non-tangential limit $\displaystyle f^*(z)=\lim_{r\to1^-}f(rz)$ exists almost everywhere on the boundary $\Gamma_X$. Moreover, for $0 < p<\infty$, the $H^p$-norm of $f$ is equivalent to the $L^p(\Gamma_X)$-norm of $f^*$ given by
\[
\|f^*\|^p_{p, \Gamma_X}=\int_{\Gamma_X}|f^*(z)|^pd\sigma(z).
\]
If $H^p(X)$ is equipped with the equivalent norm $\|f\|_{p, X} =   \left(\int_{\Gamma_{X}}|f^*(z)|^p d\sigma(z)\right)^{\frac{1}{p}},$ then $H^p(X)$ is identified as a closed subspace of $L^p(\Gamma_X)$.

\begin{example}
When $X$ is the polydisc $\D^n$ or   the unit ball $X=\B^n$ with $\Gamma_X$ being $\T^n$ or $\partial \B^n$, respectively. The  Hardy space $H^p(X)$ is isometric to the closure of the space of the radial limits  of holomorphic functions on $X$ with finite $L^p$-norm on $\Gamma_X$ (   \cite{R69,R80}). Namely,
\[
\|f\|^p_{p,X} = \sup_{0 < r <1} \int_{\Gamma_X} |f(rz)|^p d\sigma(z) =\int_{\Gamma_X}|f^*(z)|^p d\sigma(z).
\]
\end{example}

Note that $H^2(X)$ is a complete Hilbert space. The Szeg\H{o} projection is defined to be the orthogonal projection \[
S_X: L^2(\Gamma_X)\to H^2(X).
\]
By the classical theory, the Szeg\H{o} projection can be recognized as an integral operator
\[
S_X(f)(z)=\int_{\Gamma_X} K_X(z, \zeta) f(\zeta) d\sigma(\zeta)
\]
for any $f \in L^2(\Gamma_X)$, where the kernel $K_X(z, \zeta)$ is called the Szeg\H{o} kernel on $X$. If $\{e_j(z)\}$ is an orthonormal basis for $H^2(X)$, then the Szeg\H{o} kernel can be computed by
\[
K_X(z, \zeta) = \sum_j e_j(z) \overline{e_j(\zeta)}.
\]
Moreover, the Szeg\H{o} projection enjoys the following reproducing property:
$$S_X(f)(z) = \int_{\Gamma_X} K_X(z, \zeta) f(\zeta) d\sigma(\zeta)=f(z),$$ for any $f \in H^2(X)$. 
The readers may refer to \cite{R69, Ste72, K92} for details.

Recall that $G$ is the finite group of deck transformations on $X$ and $\Gamma_X$ is invariant under $G$ (Definition \ref{partbdry}). We define the following.

\begin{definition}\label{funinvariant}
A function $f$ on $X$ (or $\Gamma_X$ respectively) is said to be $G$-invariant, if it is invariant under the group actions of $G$. Namely,
$f(\tau z)=f(z)$
for all $\tau\in G$ and $z\in X$ (or $z\in\Gamma_X$ respectively).
\end{definition}

\begin{lemma}\label{kernelinvariant}
Let $G$ be a finite subgroup of the unitary group $U(n)$. The Szeg\H{o} kernel $K_X(z,\zeta)$ on $X$ is $G$-invariant in the following sense:
\[
K_X(\tau z,\tau\zeta)=K_X(z,\zeta)
\]
for any $\tau\in G$.  As a result, if $f$ is $G$-invariant, so is its Szeg\H{o} projection $S_X(f)$ on $X$.
\end{lemma}

\begin{proof}
By Corollary \ref{leb}, the induced Lebesgue measure on $\Gamma_X$ is preserved by $U(n)$. It follows that if $\{e_j\}_{j=1}^{\infty}$ is an orthonormal basis, so is $\{e_j\circ\tau\}_{j=1}^{\infty}$ for any $\tau\in G$.  Since $K_X(z,\zeta)$ is independent of the choice of the orthonormal basis,  summing over the two basis gives
\[
K_X(z,\zeta)=\sum_je_j(z)\overline{e_j(\zeta)}=\sum_je_j(\tau z)\overline{e_j(\tau\zeta)}=K_X(\tau z,\tau\zeta).
\]
As a result, if $f$ is $G$-invariant, by the change of variables $\eta=\tau\zeta$, 
\begin{align*}
S_X(f)(\tau z)
&=\int_{\Gamma_X}K_X(\tau z,\eta)f(\eta)d\sigma(\eta)\\
&=\int_{\Gamma_X}K_X(\tau z,\tau\zeta)f(\tau\zeta)d\sigma(\tau\zeta)\\
&=\int_{\Gamma_X}K_X(z,\zeta)f(\zeta)d\sigma(\zeta)=S_X(f)(z)
\end{align*}
for any $\tau\in G$.
\end{proof}

Now assume the quadruple $(X, \Omega, \Phi, G)$ satisfies Definition \ref{partbdry}. Moreover, assume $G$ to be a finite subgroup of the unitary group $U(n)$. 
Note that $\Phi$ and $\sigma$ (the induced Lebesgue measure on $\Gamma_X$) are both $G$-invariant. Writing $\Phi^*dm_{\Gamma_\Omega} = w d\sigma$, then 
$w$ is a $G$-invariant  density function on $\Gamma_X$.
We are now ready to define the Hardy spaces on $\Omega$ using the quotient structure, under an additional 
 admissible condition.

\begin{definition}\label{con}
A $G$-invariant density function $w$ on $\Gamma_X$ is admissible if there exists a non-zero $G$-invariant holomorphic function $g$ on $X$, such that the non-tangential  limit $g^*$ exists and $|g^*|^2=w$ almost everywhere on $\Gamma_X$. In such a case, the quadruple $(X, \Omega, \Phi, G)$ is also called admissible.
\end{definition}

\begin{remark}
The holomorphic function $g$ is non-zero on $X$, but its non-tangential limit $g^*$ may have zeros on $\Gamma_X$. Indeed, by Proposition \ref{den}, $w>0$ almost everywhere on $\Gamma_X$. So $g^*\neq0$ almost everywhere on $\Gamma_X$.
\end{remark}

\begin{definition}\label{Hpspace}
Let $(X, \Omega, \Phi, G)$ be an admissible quadruple. The Hardy space $H^p({\Omega})$ is defined to be the set of all holomorphic functions $f$ on $\Omega$ such that 
\begin{equation}\notag
\|f\|^p_{p, \Omega}=\frac{1}{|G|} \sup_{0 < r < 1} \int_{\Gamma_X}  |(f\circ \Phi)(rz)|^p |g(rz)|^2 d\sigma(z) < \infty, 
\end{equation}
where $|G|$ is the cardinality of $G$.
\end{definition}

\begin{theorem}\label{hardy}
Let $0< p<\infty$. For any $f \in H^p(\Omega)$, $(f\circ\Phi)^*$ exists and induces a function defined almost everywhere on $\Gamma_\Omega$, denoted by $f^*$, called the non-tangential 
limit of $f$ on $\Omega$. Moreover, the map $f \mapsto f^*$ is a quasi-isometry from $H^p(\Omega)$ onto a closed subspace of $L^p(\Gamma_\Omega)$. In particular, if $X=\D^n$ or $\B^n$, the map is an isometry.
\end{theorem}

\begin{proof}
Note that $g\neq0$ on $X$, so $g^{\frac{2}{p}}$ is a well-defined holomorphic function on $X$. For any $f\in H^p(\Omega)$,
\[
\|f\|^p_{p, \Omega}=\frac{1}{|G|} \sup_{0 < r< 1} \int_{\Gamma_X}  |g^{\frac{2}{p}}\cdot f\circ \Phi |^p(rz) d\sigma(z) < \infty
\]
implies $g^{\frac{2}{p}}\cdot f\circ \Phi\in H^p(X)$. By the classical theory in Hardy space, the holomorphic function $g^{\frac{2}{p}}\cdot f\circ \Phi$ admits a non-tangential limit $(g^{\frac{2}{p}}\cdot f\circ \Phi)^*$ on $\Gamma_X$. Moreover, 
\begin{equation}\label{h2iso}
\begin{split}
\|f\|^p_{p, \Omega}
&\approx \frac{1}{|G|} \int_{\Gamma_X}|(g^{\frac{2}{p}}\cdot f\circ\Phi)^*(z)|^pd\sigma(z)\\
&=\frac{1}{|G|}\int_{\Gamma_X}|(f\circ\Phi)^*(z)|^pw(z)d\sigma(z).
\end{split}
\end{equation}

For any $\tau \in G, z \in \overline{X}$, since $\Phi \circ \tau(z) = \Phi(z)$ and $r\tau(z) = \tau(rz)$  for any $z \in \overline{X}$, then 
$$(f\circ\Phi)^*(\tau(z)) = \lim_{r\rightarrow 1^-} (f\circ\Phi)(r\tau(z)) =  \lim_{r\rightarrow 1^-} (f\circ\Phi)(\tau(rz)) =  \lim_{r\rightarrow 1^-} (f\circ\Phi)(rz) =(f\circ\Phi)^*(z) .$$
For any $z \in \Gamma_\Omega$, 
the non-tangential limit of $f$ on $\Omega$ can be defined as  $$f^*(z)=\frac{1}{|G|} \left( \Phi_*(f\circ\Phi)^*\right)(z) = (f\circ\Phi)^*(\xi), $$ 
given $\Phi(\xi)=z$, where $\Phi_*$ is the push forward in the sense of distribution. Therefore, by  \eqref{change},
$$\|f\|^p_{p, \Omega}
\approx  \int_{\Gamma_{\Omega}}|f^*(z)|^p dm_{\Gamma_{\Omega}}(z),$$
implying that the map $f \mapsto f^*$ is a quasi-isometry from $H^p(\Omega)$ onto a closed subspace of $L^p(\Gamma_\Omega)$.
In particular, $X=\D^n$ or $\B^n$, the map is an isometry.
\end{proof}

\begin{remark}
If we only consider $p=2$, the assumption $g\neq0$ on $X$ is not necessary.
\end{remark}

Similar to $H^p(X)$, the space $H^p(\Omega)$ is equipped with the equivalent norm:
\[
\|f\|_{p, \Omega} =   \left(\int_{\Gamma_{\Omega}}|f^*(z)|^p dm_{\Gamma_{\Omega}}(z)\right)^{\frac{1}{p}}
\]
for any $f \in H^p(\Omega)$. As a result, $H^p(\Omega)$ is identified as a closed subspace of $L^p(\Gamma_\Omega)$ given in Theorem \ref{hardy}, and $H^2(\Omega)$ is a complete Hilbert space. 

\begin{definition}\label{defszego}
The Szeg\H{o} projection $S_{\Omega}$ on $\Omega$ is the orthogonal projection from $L^2(\Gamma_\Omega)$ onto  $H^2({\Omega})$.
\end{definition}

\begin{remark}
Classical Hilbert space theory shows that the Szeg\H{o} projection $S_{\Omega}$ is an integral operator
\[
S_{\Omega}(f)(z) = \int_{\Gamma_\Omega} K_\Omega(z, \zeta) f(\zeta) dm_{\Gamma_{\Omega}}(\zeta),
\]
where $K_\Omega(z, \zeta)$ is called the Szeg\H{o} kernel on $\Omega$. If $\{\psi_j(z)\}$ is an orthonormal basis for $H^2(\Omega)$, then the Szeg\H{o} kernel can be computed by
\[
K_\Omega(z, \zeta) = \sum_j \psi_j(z) \overline{\psi_j(\zeta)},
\]
where the convergence holds uniformly on any compact subsets in $\Omega \times \Omega$. Moreover, the Szeg\H{o} projection $S_{\Omega}$ is self-adjoint in $L^2(\Gamma_{\Omega})$ and enjoys the interpolation properties in $L^p(\Gamma_{\Omega})$.
\end{remark}

We now derive an integral representation of $S_{\Omega}$, which resembles Bell's transformation formula for the Bergman projection (\cite{Bel82}). 

\begin{theorem}\label{invariant}
Let $K_X(z,\zeta)$ be the Szeg\H{o} kernel on $X$ with $(z,\zeta)\in X\times\Gamma_X$. Then the Szeg\H{o} projection $S_{\Omega}$ can be expressed as
\[
S_{\Omega}(f)(z)=\frac{1}{|G|} \Phi_*\left( \int_{\Gamma_X}\frac{K_X(z,\zeta)}{g(z)\overline{g^*(\zeta)}}f(\Phi(\zeta))w(\zeta)d\sigma(\zeta)\right)
\]
for $f\in L^2(\Gamma_{\Omega})$.
\end{theorem}

\begin{proof}
Define
\begin{align*}
T(f)(z)&:=\frac{1}{|G|} \Phi_*\left( \int_{\Gamma_X}\frac{K_X(z,\zeta)}{g(z)\overline{g^*(\zeta)}}f(\Phi(\zeta))w(\zeta)d\sigma(\zeta) \right) \\
&=\frac{1}{|G|} \Phi_*\left( \frac{1}{g(z)}\int_{\Gamma_X}K_X(z,\zeta)f(\Phi(\zeta))g^*(\zeta)d\sigma(\zeta)\right)\\
&=\frac{1}{|G|} \Phi_*\left(\frac{1}{g(z)}S_X(g^*\cdot f\circ\Phi)(z)\right)
,
\end{align*}
where $S_X$ is the Szeg\H{o} projection on $X$. We show that $T$ projects from $L^2(\Gamma_\Omega)$ onto $H^2({\Omega})$ orthogonally.

For $f\in L^2(\Gamma_{\Omega})$, by \eqref{h2iso} with $p=2$, one sees $g^*\cdot f\circ\Phi\in L^2(\Gamma_X)$ and it is $G$-invariant. Then $S_X(g^*\cdot f\circ\Phi)\in H^2(X)$ and it is $G$-invariant by Lemma \ref{kernelinvariant}. As a result, $T(f)\circ \Phi=\frac{1}{g} \cdot S_X(g^*\cdot f\circ\Phi)\in H^2({\Omega})$ by Definition \ref{Hpspace}. So $T$ maps $L^2(\Gamma_{\Omega})$ to $H^2({\Omega})$.

For $f\in H^2({\Omega})$, again one sees $g\cdot f\circ\Phi\in H^2(X)$. Then $S_X(g\cdot f\circ\Phi)=g\cdot f\circ\Phi$, and hence $T(f)=\frac{1}{|G|} \Phi_*\left( \frac{ S_X(g\cdot f\circ\Phi)}{g}\right)=\frac{1}{|G|} \Phi_*(f\circ \Phi) =f$ on $\Omega$. So $T$ preserves $H^2({\Omega})$, or equivalently $T$ is a projection.

To see the orthogonality, one needs to verify $(f-T(f))\perp H^2(\Omega)$ in $L^2(\Gamma_{\Omega})$ for any $f\in L^2(\Gamma_{\Omega})$. This is equivalent to show $\langle f,h\rangle_{\Gamma_{\Omega}}=\langle T(f),h\rangle_{\Gamma_{\Omega}}$ for all $h\in H^2(\Omega)$.  Note that
\[
\langle T(f),h\rangle_{\Gamma_{\Omega}}=\frac{1}{|G|}\langle S_X(g^*\cdot f\circ\Phi), g^*\cdot h\circ\Phi\rangle_{\Gamma_X}=\frac{1}{|G|}\langle g^*\cdot f\circ\Phi, g^*\cdot h\circ\Phi\rangle_{\Gamma_X},
\]
since $S_X$ is orthogonal from $L^2(\Gamma_X)$ onto $H^2(X)$ and $g^*\cdot h\circ\Phi\in H^2(X)$ by \eqref{h2iso}. On the other hand, the change of variables shows
\[
\langle f,h\rangle_{\Gamma_{\Omega}}=\frac{1}{|G|}\langle g^*\cdot f\circ\Phi, g^*\cdot h\circ\Phi\rangle_{\Gamma_X}.
\]
This completes the proof.
\end{proof}

\begin{remark}
The formula in Theorem \ref{invariant} is an integral representation of $S_{\Omega}$ lifted on $X$. The kernel inside may not be the Szeg\H{o} kernel on $\Omega$, since the elements of the basis may not be $G$-invariant. Lemma \ref{kernelinvariant} states that the Szeg\H{o} projection on the covering space $X$ preserves $G$-invariant subspaces.
\end{remark}

To study the regularity of $S_{\Omega}:L^p(\Gamma_{\Omega})\to H^p(\Omega)$, by Theorem \ref{invariant}, we are able to consider the regularity of the Szeg\H{o} projection on weighted $L^p$ space over $X$ instead.

\begin{theorem}\label{main1}
Let $S_{\Omega}$ and $S_{X}$ be the Szeg\H{o} projections on $\Omega$ and $X$, respectively. For $1<p<\infty$, $S_{\Omega}:L^p(\Gamma_{\Omega})\to H^p(\Omega)$ is bounded if and only if $S_{X}:L^p(\Gamma_X,w^{1-\frac{p}{2}})\to H^p(X,w^{1-\frac{p}{2}})$ is bounded restricted to  the subset of $G$-invariant functions.
\end{theorem}

\begin{proof}
For $f\in L^p(\Gamma_{\Omega})$, then $h=g^*\cdot f\circ\Phi$ is a $G$-invariant function on $\Gamma_X$. The $L^p(\Gamma_{\Omega})$-norm of $f$ is
\[
|G| \cdot \int_{\Gamma_{\Omega}}|f|^pdm_{\Gamma_{\Omega}}=\int_{\Gamma_X}|g^*\cdot f\circ\Phi|^pw^{1-\frac{p}{2}}d\sigma=\int_{\Gamma_X}|h|^pw^{1-\frac{p}{2}}d\sigma,
\]
which is the $L^p(\Gamma_X,w^{1-\frac{p}{2}})$-norm of $h$.
 By Theorem \ref{invariant}, $S_{\Omega}(f)$ can be written as
\[
S_{\Omega}(f)=\frac{1}{|G|} \Phi_* \left( \frac{1}{g(z)}\int_{\Gamma_X}K_X(z,\zeta)h(\zeta)d\sigma(\zeta)\right)=\frac{1}{|G|} \Phi_* \left(\frac{1}{g(z)}S_X(h)(z) \right),
\]
whose $H^p(\Omega)$-norm by \eqref{h2iso} can be written as
\[
|G| \cdot\|S_{\Omega}(f)\|_{p,\Omega}^p= \int_{\Gamma_X}|g^*(z)|^{-p}|S_X(h)(z)|^pw(z)d\sigma(z)=\int_{\Gamma_X}|S_X(h)|^pw^{1-\frac{p}{2}}d\sigma.
\]
This is the $L^p(\Gamma_X,w^{1-\frac{p}{2}})$-norm of $S_X(h)$. Therefore,
\[
\|S_{\Omega}(f)\|_{p,\Omega}^p\le C\int_{\Gamma_{\Omega}}|f|^pdm_{\Gamma_{\Omega}}
\]
if and only if
\[
\int_{\Gamma_X}|S_X(h)|^pw^{1-\frac{p}{2}}d\sigma\le C\int_{\Gamma_X}|h|^pw^{1-\frac{p}{2}}d\sigma.
\]
\end{proof}

 \begin{remark}
When $\Gamma_{\Omega}$ is at least $C^2$-smooth, the density function $w$ is bounded away from zero on $\Gamma_X$. With the choice of $g\in H^{\infty}(X)$, the Hardy spaces $H^p(\Omega)$ defined in Definition \ref{Hpspace} and the Szeg\H{o} projection $S_{\Omega}$ defined in Definition \ref{defszego} are compatible with those defined in the classical sense.
\end{remark}

\begin{remark}
While the Hardy spaces $H^p(\Omega)$ and the Szeg\H{o} projection $S_{\Omega}$ depend on  the choice of $g$, the $L^p$ regularity of the Szeg\H{o} projection actually does not. As a result, when studying the $L^p$ regularity of the Szeg\H{o} projection, it is sufficient to consider one $g$.
\end{remark}


\section{Quotient domains of the polydisc}\label{SecDn}

In this section, the theory in \S\ref{generalsetup} is applied to the model domain $X=\D^n$. In the first subsection, by imposing an additional restriction on $G$, we derive a sufficient condition on the density function in order for the quadruple $(\D^n, \Omega, \Phi, G)$ to be admissible. In the second subsection, we apply the theory to study the $L^p$ regularity of the Szeg\H{o} projection on the symmetrized bidisc.

\subsection{Admissible quadruple for quotient domains of $\D^n$}

\begin{definition}
An element $\tau\in G$ is a rotation if there is a $\beta=(\beta_1,\dots,\beta_n)$ such that
\[
\tau z=(e^{i\beta_1}z_1,\dots,e^{i\beta_n}z_n)
\]
for any $z=(z_1, \dots, z_n)\in\C^n$.
\end{definition}

\begin{remark}
Since $G$ is finite, $\beta_j=\frac{2k_j\pi}{|G|}$ for some $k_j\in\Z^+$ for $j=1,\dots,n$.
\end{remark}

\begin{definition}
An element $\tau\in G$ is a permutation if there is a permutation $\kappa$ of $\{1,\dots,n\}$ such that
\[
\tau z=(z_{\kappa(1)},\dots,z_{\kappa(n)})
\]
for any $z\in\C^n$.
\end{definition}

Assume that the quadruple $(\D^n, \Omega, \Phi, G)$ satisfies Definition \ref{partbdry}.  In this section, we furthermore assume $G$ to be a finite subgroup of $U(n)$ generated by rotations and permutations.

\begin{lemma}\label{FseriesGin}
The Fourier series of a $G$-invariant integrable function on $\T^n$ must also be $G$-invariant. 
\end{lemma}

\begin{proof}
Assume that $f$ is invariant under a permutation $\tau_1$. Namely, $f(\tau_1 \theta)=f(\theta)$ for any $\theta\in \T^n$. Its Fourier coefficient
\[
\hat f(k)=\int_{\T^n}f(\theta)e^{-ik\cdot\theta}d\theta
\]
is also invariant under the same permutation on $\Z^n$. As a result, its Fourier series
\[
S_f(\theta)=\sum_{k\in\Z^n}\hat f(k)e^{ik\cdot\theta}
\]
is invariant under $\tau_1$.

Assume $f$ is invariant under a rotation $\tau_2$. Namely, $f(\theta+\beta)=f(\theta)$ for some $\beta\in\T^n$. Then its Fourier coefficients satisfy
\[
\hat f(k)=\int_{\T^n}f(\theta)e^{-ik\cdot\theta}d\theta=\int_{\T^n}f(\theta+\beta)e^{-ik\cdot\theta}d(\theta+\beta)=e^{ik\cdot\beta}\hat f(k).
\]
So its Fourier series satisfies
\[
S_f(\theta)=\sum_{k\in\Z^n}\hat f(k)e^{ik\cdot\theta}=\sum_{k\in\Z^n}e^{ik\cdot\beta}\hat f(k)e^{ik\cdot\theta}=S_f(\theta+\beta),
\]
which implies the invariance of $S_f(\theta)$ under $\tau_2$. This completes the proof, since $G$ is generated by rotations and permutations.
\end{proof}

The next result is one of the key theorems in the paper and it provides a sufficient condition for the study of the Szeg\H{o} projection on quotient domains of $\D^n$.
More precisely, it reduces the verification of the quatertriple $(\D^n, \Omega, \Phi, G)$ to be admissible to an integrability condition on the density function. The proof is based on an idea in \cite{R69}.

\begin{theorem}\label{normsquare}
Let $w$ be a real-valued, positive almost everywhere, $G$-invariant, continuous function on $\T^n$. If $\log w\in L^1(\T^n)$,  then $w$ is admissible. 
\end{theorem}

\begin{proof}
Let $P$ be the Poisson integral on $\T^n$.
For the extended real-valued $G$-invariant continuous function $\frac{1}{2}\log w$, we claim that there is a real $G$-invariant singular measure $\mu$ on $\T^n$ (with respect to the Lebesgue measure) such that
\[
u:=P\left(\frac{1}{2}\log w-d\mu\right)
\]
is the real part of some holomorphic function $h$ on $\D^n$. By \cite[Theorem 2.3.1]{R69}, the radial limit $u^*=\frac{1}{2}\log w$ almost everywhere on $\T^n$.

To construct $\mu$, one follows the argument from \cite[Theorem 2.4.2 and Exercise 2.4.3]{R69}. For a $G$-invariant nonnegative trigonometric polynomial $\varphi$ on $\T^n$, $\hat \varphi(k)=0$ except for $k$ in some finite set $S\subset\Z^n$.  Let $v$ be a sufficiently large multiple of $|G|\cdot(1,\dots,1)$ in $\Z^n$, so that
\[
0\notin S-jv\subset Y_n:=\Z_+^n\cup(-\Z_+^n)
\]
for $j=\pm1,\pm2,\dots$, where $\Z_+^n$ is the $n$-Cartesian product of non-negative integers. Let $H$ be the subgroup of $\T^n$ consisting all $\theta$ for which $e^{iv\cdot\theta}=1$. Let $\mu_H$ be the Haar measure of $H$ and $d\mu=\varphi d\mu_H$. So $\mu\ge0$ is singular, since $H$ is in lower dimension.  By the choice of $v$, $H$ is invariant by any permutation and rotation. So the Haar measure $\mu_H$ and hence $\mu$ is $G$-invariant.  Moreover, since $\varphi$ is a trigonometric polynomial,
\[
\hat \mu(m)=\sum_{k\in S}\hat \varphi(k)\int_{\T^n}e^{i(k-m)\cdot\theta}d\mu_H.
\]
Note that by construction of $H$ and $\mu_H$,
\[
\int_{\T^n}e^{i(k-m)\cdot\theta}d\mu_H=\left\{\begin{array}{cl} 1, & \text{if}\,\,\,k-m=jv\,\,\,\text{for some}\,\,\,j\in\Z \\ 0, & \text{otherwise} \end{array}\right..
\]
Therefore,
\[
\hat \mu(m)=\sum_{j\in\Z}\hat \varphi(m+jv)
\]
for $m+jv\in S$. So when $m\notin Y_n$, it must be $j=0$, which gives $\hat \mu(m)=\hat\varphi(m)$ and $\|\mu\|=\hat\mu(0)=\hat\varphi(0)=\|\varphi\|_{L^1}$. 

Now look at $\frac{1}{2}\log w\in L^1(\T^n)$, and write $\frac{1}{2}\log w=\phi_0-\phi_1$, where $\phi_0=1+\frac{1}{2}\log^+w$ is $G$-invariant, positive, continuous and $\phi_1=1+\frac{1}{2}\log^-w$ is $G$-invariant, integrable, continuous with values in $\R^+\cup\{\infty\}$. As in the proof of \cite[Theorem 2.4.2 and Exercise 2.4.3]{R69}, write $\phi_1=\sum_jp_j$, where each $p_j$ is a $G$-invariant positive continuous function. This can be done: since $\phi_1$ is $G$-invariant, the push-forward $\Phi_*(\phi_1)$ is a positive lower semicontinuous function on $\Gamma_{\Omega}$; it can be written as a sum of positive continuous functions by a theorem of Baire on $\Gamma_{\Omega}$; the $G$-invariant $p_j$ are obtained by the pulling the functions back on $\T^n$ via $\Phi^*$. Moreover, for each $p_j$, it can be written as $p_j=\sum_ip_{ji}$, where by Lemma \ref{FseriesGin} each $p_{ji}$ is a $G$-invariant positive trigonometric polynomial. Hence, 
\[
\phi_1=\sum_jp_j=\sum_i\left(\sum_jp_{ji}\right)=\sum_i\varphi_i,
\]
a sum of positive trigonometric polynomials. For each $\varphi_i$, follow the preceding construction to obtain a nonnegative $G$-invariant singular measure $\mu_i$. Since 
\[
\sum_i\|\mu_i\|=\sum_i\|\varphi_i\|_{L^1}=\|\phi_1\|_{L^1}<\infty,
\]
the sume of the measures $\sum_i\mu_i$ converges to a nonnegative $G$-invariant singular measure $\lambda_1$. A limit argument shows
\[
\hat \lambda_1(m)=\sum_{i}\hat\mu_i(m)=\sum_{i}\hat\varphi_i(m)=\hat \phi_1(m),
\]
for $m\notin Y_n$. Hence the Fourier coefficients of $\phi_1-d\lambda_1$ are all 0 outside $Y_n$. Similarly, $\phi_0$ can be written as a sum of positive trigonometric polynomials, and there is a nonnegative $G$-invariant singular measure $\lambda_0$ for $\phi_0$, which satisfies the same property.

Let $\mu=\lambda_0-\lambda_1$ and $f=\frac{1}{2}\log w-d\mu$ on $\T^n$. Since $w$ and $\mu$ are $G$-invariant, so is $f$. By the definition of Poisson integral,
\[
u=P(f)=\sum_{k\in Y_n}\hat f(k)r_1^{|k_1|}\cdots r_n^{|k_n|}e^{ik\cdot\theta},
\]
where $\hat f$ is the Fourier coefficient of $f$.  Same argument as in Lemma \ref{FseriesGin} shows that $u$ is $G$-invariant. Following the argument in \cite[Theorem 2.1.4]{R69}, one defines the $G$-invariant holomorphic function
\[
h=\hat f(0)+\sum_{0\neq k\in\Z^n_+}[\hat f(k)+\overline{\hat f(-k)}]r_1^{k_1}\cdots r_n^{k_n}e^{ik\cdot\theta}=\hat f(0)+\sum_{0\neq k\in\Z^n_+}2\hat f(k)z^k
\]
satisfying $\Re(h)=u$. Here $f$ is real, so $\hat f(k)=\overline{\hat f(-k)}$. 

Let $g=e^h$. Then $g$ is the desired non-zero $G$-invariant holomorphic function on $\D^n$, since $|g|=e^{\Re(h)}=e^u$.
\end{proof}

\subsection{The Szeg\H{o} projection on the symmetrized bidisc}


Let $X=\D^2$, $\Phi(z)=(z_1+z_2,z_1z_2)$, and $\Omega=\Sb^2$ be the symmetrized bidisc. Consider $\Gamma_{\D^2}=\T^2$ the distinguished boundary of the bidisc, then $\Gamma_{\Sb^2}=\Phi(\T^2)$ is the induced boundary. The group $G$ is simply $S_2$, the permutations of $(z_1,z_2)$.

For $j=1,2$, $z_j=e^{i\theta_j}$ on $\T$, where $\theta_j\in[-\pi,\pi]$. Then the straightforward calculation yields
\[
J_{\R}Z=\left(\begin{array}{cc} \frac{\partial z_1}{\partial \theta_1} & \frac{\partial z_2}{\partial \theta_1} \\ \frac{\partial z_1}{\partial \theta_2} & \frac{\partial z_2}{\partial \theta_2} \end{array}\right)=\left(\begin{array}{cc} iz_1 \\ & iz_2 \end{array}\right),
\]
\[
J_{\C}\Phi=\left(\begin{array}{cc} \frac{\partial \phi_1}{\partial z_1} & \frac{\partial \phi_2}{\partial z_1} \\ \frac{\partial \phi_1}{\partial z_2} & \frac{\partial \phi_2}{\partial z_2} \end{array}\right)=\left(\begin{array}{cc} 1 & z_2 \\ 1 & z_1 \end{array}\right),
\]
and by \eqref{matrixA}
\[
A=\left(\begin{array}{cc} iz_1 \\ & iz_2 \end{array}\right)\cdot\left(\begin{array}{cc} 1 & z_2 \\ 1 & z_1 \end{array}\right)=\left(\begin{array}{cc} iz_1 & iz_1z_2 \\ iz_2 & iz_1z_2 \end{array}\right),
\]
implying
\[
A\overline{A^t}=\left(\begin{array}{cc} iz_1 & iz_1z_2 \\ iz_2 & iz_1z_2 \end{array}\right)\left(\begin{array}{cc} -i\overline{z_1} & -i\overline{z_2} \\ -i\overline{z_1}\overline{z_2} & -i\overline{z_1}\overline{z_2} \end{array}\right)=\left(\begin{array}{cc} 2 & z_1\overline{z_2}+1 \\ \overline{z_1}z_2+1 & 2 \end{array}\right)
\]
and
\[
\Re(A\overline{A^t})=\left(\begin{array}{cc} 2 & 1+\cos(\theta_1-\theta_2) \\ 1+\cos(\theta_1-\theta_2) & 2 \end{array}\right).
\]
Hence, by Proposition \ref{den}, one obtains the density function 
\begin{equation}
\label{bidiscJ}
\begin{split}
w=\sqrt{\det[\Re(A\overline{A^t})]}
&=\sqrt{4-[1+\cos(\theta_1-\theta_2)]^2}\\
&=\sqrt{2-2\cos(\theta_1-\theta_2)+\sin^2(\theta_1-\theta_2)}.
\end{split}
\end{equation}

\begin{proposition}
\label{bidiscallow}
The density function $w=\sqrt{2-2\cos(\theta_1-\theta_2)+\sin^2(\theta_1-\theta_2)}$ on $\T^2$ is admissible. 
\end{proposition}

\begin{proof}
For $x\in\T$, let
\[
u(x)=\sqrt{2-2\cos x+\sin^2x}=\sqrt{3-2\cos x-\cos^2x}.
\]
Then $w(\theta_1,\theta_2)=u(\theta_1-\theta_2)$ is $G$-invariant and
\[
\int_{\T^2}|\log w|d\theta_1d\theta_2=\int_{\T^2}|\log u(\theta_1-\theta_2)|d\theta_1d\theta_2=\int_{\T}d\theta_2\int_{-\pi}^{\pi}|\log u(x)|dx.
\]

Note that by continuity, $\log u$ is integrable when $x$ is away from $0$. By the Taylor expansion of cosine about $0$, one can verify that $\log u(x)\approx\log |x|$ is integrable when $x$ is around $0$.  So $\log w\in L^1(\T^2)$, and $w$ is admissible by Theorem \ref{normsquare}.
\end{proof}

We recall some results in order to prove the $L^p$ regularity of the Szeg\H{o} projection $S_{\Sb^2}$. We also call the the density function with respect to the volume form $d\theta_1d\theta_2$ the weight function. 

\begin{lemma}
\label{compareweight}
Let $A$ be a subset of $L^p(\Gamma_X,w^{1-\frac{p}{2}})$.
If $w\approx\tilde{w}$, then $S_X$ is bounded on the set $A$  if and only if it is bounded on $A$ in $L^p(\Gamma_X,\tilde{w}^{1-\frac{p}{2}})$.
\end{lemma}
\begin{proof}
Since the two weights are comparable, the two $L^p$ spaces are the same. Moreover,
\[
\int_{\Gamma_X}|S_X(f)|^pw^{1-\frac{p}{2}}d\sigma\lesssim\int_{\Gamma_X}|S_X(f)|^p\tilde{w}^{1-\frac{p}{2}}d\sigma\lesssim\int_{\Gamma_X}|f|^p\tilde{w}^{1-\frac{p}{2}}d\sigma\lesssim\int_{\Gamma_X}|f|^pw^{1-\frac{p}{2}}d\sigma,
\]
for any $f \in A$. The other direction follows from a similar argument.
\end{proof}

\begin{definition}
Use $I$ to denote intervals in $\T$.
For $1<p<\infty$, a weight $\mu$ on $\T$ is in $A_p(\T)$ if
\[
\sup_{I}\left(\frac{1}{|I|}\int_I\mu(\theta)d\theta\right)\left(\frac{1}{|I|}\int_I\mu(\theta)^{-\frac{1}{p-1}}d\theta\right)^{p-1}<\infty.
\]
\end{definition}

\begin{theorem}[Theorem 3 in \cite{MZ15}]
\label{AponT}
For $1<p<\infty$, the Szeg\H{o} projection $S_{\D}$ on $\D$ is bounded on $L^p(\T,\mu)$ if and only if $\mu\in A_p(\T)$.
\end{theorem}

\begin{theorem}[Proposition 4 and Theorem 5 in \cite{MZ15}]
\label{intervalofponT}
For $\alpha\ge0$, $|z-1|^{\alpha(2-p)}\in A_p(\T)$ if and only if $p\in\left(\frac{2\alpha+1}{\alpha+1},\frac{2\alpha+1}{\alpha}\right)$.
\end{theorem}

By Proposition \ref{bidiscallow}, $(\D^2, \Sb^2, \Phi, S_2)$ is an admissible quadruple. Hence, we are able to apply Theorem \ref{main1} to derive the $L^p$-regularity of the Szeg\H{o} projection on $\Sb^2$.

\begin{theorem}
The Szeg\H{o} projection on the symmetrized bidisc $\Sb^2$ is bounded on $L^p(\Gamma_{\Sb^2})$ if and only if $\frac{4}{3}<p<4$.
\end{theorem}

\begin{proof}
Note that on $\T^2$
\[
\sin^2(\theta_1-\theta_2)=1-\cos^2(\theta_1-\theta_2)\le 2-2\cos(\theta_1-\theta_2)=|z_1-z_2|^2.
\]
It follows that
\[
|z_1-z_2|\le w\le\sqrt{2}|z_1-z_2|,
\]
and thus
\[
w=\sqrt{2-2\cos(\theta_1-\theta_2)+\sin^2(\theta_1-\theta_2)}\approx|z_1-z_2|\approx|\theta_1-\theta_2|
\]
on $\T^2$. By Lemma \ref{compareweight}, in order to show that $S_{\D^2}$ is bounded on  the set of $G$-invariant functions in $L^p(\T^2, w^{1-\frac{p}{2}})$
it suffices to check the integral operator $S_{\D^2}$ on the set of $G$-invariant functions in $L^p(\T^2,|z_1-z_2|^{1-\frac{p}{2}})$.

To see the boundedness, one follows the idea in \cite{CKY20}. Since $S_{\D^2}$ is an iterated integral, by the symmetry of $z_1$ and $z_2$,  one only needs to check $S_{\D}$ on $L^p(\T,|z_1-z_2|^{1-\frac{p}{2}})$, where $z_1$ is considered as a variable and $z_2$ is considered as a parameter. By Theorem \ref{AponT}, it is sufficient to verify
$$|z_1-z_2|^{1-\frac{p}{2}}\in A_p(\T)$$
with a uniform constant independent of the paramenter $z_2$ in $\T$. By the rotational symmetry on $\T$, it suffices to verify
\[
|z-1|^{1-\frac{p}{2}}\in A_p(\T).
\]
By Theorem \ref{intervalofponT} with $\alpha=\frac{1}{2}$, the weight $|z-1|^{1-\frac{p}{2}}$ belongs to $A_p(\T)$ when
$\frac{4}{3}<p<4.$ This shows that $S_{\D^2}$ is bounded on $L^p(\T^2,|z_1-z_2|^{1-\frac{p}{2}})$ if $\frac{4}{3}<p<4$.

To obtain the unboundedness for $p=4$, consider the $G$-invariant test function
$h(z_1,z_2)=|z_1-z_2|^2$. On one hand, 
\[
\|h\|^4_{L^4(\T^2,w^{1-\frac{4}{2}})}=\int_{\T^2}|h|^4w^{-1}d\sigma\approx\int_{\T^2}|\theta_1-\theta_2|^{7}d\theta_1d\theta_2<\infty.
\]
 On the other hand, one obtains
\[
S_{\D^2}(h)(z)=\int_{\T^2}\frac{|\zeta_1|^2+|\zeta_2|^2-\zeta_1\overline{\zeta_2}-\overline{\zeta_1}\zeta_2}{(1-z_1\overline{\zeta_1})(1-z_2\overline{\zeta_2})}d\sigma(\zeta)=2,
\]
and thus
\[
\|S_{\D^2}(h)\|^4_{L^4(\T^2,w^{1-\frac{4}{2}})}=\int_{\T^2}2^4\cdot w^{-1}d\sigma\approx\int_{\T^2}|\theta_1-\theta_2|^{-1}d\theta_1d\theta_2=\infty.
\]
This implies that $S_{\D^2}$ is not bounded on the set of $G$-invariant functions in $L^4(\T^2, w^{-1})$. Therefore, by Theorem \ref{main1}, 
$S_{\Sb^2}$ is not bounded on $L^4(\Gamma_{\Sb^2})$. This completes the proof as $S_{\Sb^2}$ is self-adjoint.
\end{proof}

\section{Quotient domains of the ball}\label{SecBn}

In this section, the theory in \S\ref{generalsetup} is applied to the model domain $X=\B^n$. In the first subsection, we derive a sufficient condition on the density function in order for the quadruple $(\B^n, \Omega, \Phi, G)$ to be admissible. After that, we apply the theory to study the $L^p$ regularity of the Szeg\H{o} projection on generalized Thullen domains and the minimal ball.

\subsection{Admissible quadruple for quotient domains of $\B^n$}

Throughout this section, the quadruple $(\B^n,\Omega, \Phi, G)$ is assumed to satisfy Definition \ref{partbdry}. Furthermore, $G$ is assumed to be a finite subgroup of $U(n)$. Parallel to Theorem \ref{normsquare}, we first derive a key theorem that
provides a sufficient condition for the quadruple $(\B^n, \Omega, \Phi, G)$ to be admissible.

\begin{theorem}\label{normsquareb}
Let $w$ be a real-valued, positive almost everywhere, $G$-invariant, continuous function on $\Sp^{2n-1}$. If $\log w\in L^1(\Sp^{2n-1})$, then $w$ is admissible.
\end{theorem}

\begin{remark}
Similar to Theorem \ref{normsquare}, there is a real singular measure $\mu$ on $\Sp^{2n-1}$ and a holomorphic function $h$ such that
\[
\Re(h)=P\left(\frac{1}{2}\log w-d\mu\right),
\]
where $P$ is the Poisson integral on $\Sp^{2n-1}$. Hence $\Re(h^*)=\frac{1}{2}\log w$ almost everywhere on $\Sp^{2n-1}$, see \cite{R80}. The existence of $\mu$ and $h$ follows essentially from Aleksandrov's Modification Theorem, see \cite{Ale82}. The argument below follows from \cite[Theorem 3.3, Lemma 3.4, 3.7, and 3.8]{R83}, which is a special case with exponent $p=\frac{1}{2}$ and degree $m=1$.

\end{remark}

\begin{definition}\label{metricball}
For $\zeta,\eta\in\Sp^{2n-1}$, define a metric
\[
d(\zeta,\eta):=|1-\langle\zeta,\eta\rangle|^{\frac{1}{2}},
\]
and the metric ball centered at $\eta$ of radius $\delta$ by
\[
Q_{\delta}(\eta):=\{\zeta\in\Sp^{2n-1}:\, d(\zeta,\eta)<\delta\},
\]
for $0<\delta\le\sqrt{2}$. The notation $\sigma(Q_{\delta}(\eta))$ denotes the volume of $Q_{\delta}(\eta)$ with respect to the Lebesgue measure $\sigma$ on $\Sp^{2n-1}$.
\end{definition}

\begin{remark}
It is well-known that $d$ is indeed a metric on $\Sp^{2n-1}$ (  \cite[Proposition 5.1.2]{R80}). Also, the metric balls together with the measure $\sigma$ on $\Sp^{2n-1}$ satisfy the engulfing property and the doubling property in measure theory. See Proposition \ref{doublemeasure} below for $n=2$ and the proof for general $n$ is similar (  \cite[Proposition 5.1.4]{R80}).
\end{remark}

\begin{lemma}\label{metricbinner}
There is a constant $\alpha<1$ with the following property: to every metric ball $Q=Q_{\delta}(\eta)$, corresponds an $f$ such that
\begin{enumerate}
\item $f(z)=H(\langle z,\eta\rangle)$, where $\eta$ is the center of $Q$ and $H$ is holomorphic in $\D$ and continuous on $\overline{\D}$ with $H(0)=0$;
\item $\int_{\Sp^{2n-1}}|f|^{\frac{1}{2}}d\sigma<\sigma(Q)$;
\item $\int_{\Sp^{2n-1}}|\chi_Q-\Re(f)|^{\frac{1}{2}}d\sigma<\alpha\cdot \sigma(Q)$, where $\chi_Q$ is the characteristic function of $Q$ on $\Sp^{2n-1}$.
\end{enumerate}
\end{lemma}

\begin{proof}
This is exactly \cite[Lemma 3.7]{R83} for the case $m=1$.
\end{proof}

\begin{lemma}\label{metricbinner2}
There is a constant $\beta<1$ with the following property: for $0<\tau<1$ and $Q=Q_r(\zeta_0)$, letting $\delta=\tau r$ and $\{\eta_1,\dots,\eta_N\}$ be a maximal subset of $Q_{\frac{r}{2}}(\zeta_0)$ with respect to having $d(\eta_j,\eta_k)\ge2\delta$ for all $j\neq k$ and letting $f_{\eta_1},\dots,f_{\eta_N}$ be the functions from Lemma \ref{metricbinner} associated to $Q_{\delta}(\eta_1),\dots,Q_{\delta}(\eta_N)$, then $f=f_{\eta_1}+\cdots+f_{\eta_N}$ satisfies
\begin{enumerate}
\item $|f|<1$ on $Q$;
\item $|f|<\tau$ on $\Sp^{2n-1}\setminus Q$;
\item $\int_{\Sp^{2n-1}}|f|^{\frac{1}{2}}d\sigma<\sigma(Q)$;
\item $\int_{\Sp^{2n-1}}|\chi_Q-\Re(f)|^{\frac{1}{2}}d\sigma<\beta\cdot \sigma(Q)$.
\end{enumerate}
\end{lemma}

\begin{proof}
This is exactly \cite[Lemma 3.8]{R83} for the case $m=1$.
\end{proof}

\begin{lemma}\label{mainlemball}
There is a constant $\gamma<1$ with the following property: to every positive $G$-invariant $\varphi\in C(\Sp^{2n-1})$, corresponds an $F$ such that
\begin{enumerate}
\item $F$ is holomorphic in $\B^n$, continuous on $\overline{\B^n}$, $F(0)=0$, and $F$ is $G$-invariant,
\item $\varphi-\Re(F)>0$ on $\Sp^{2n-1}$,
\item $\int_{\Sp^{2n-1}}|F|^{\frac{1}{2}}d\sigma<\int_{\Sp^{2n-1}}\varphi^{\frac{1}{2}}d\sigma$,
\item $\int_{\Sp^{2n-1}}(\varphi-\Re(F))^{\frac{1}{2}}d\sigma<\gamma\int_{\Sp^{2n-1}}\varphi^{\frac{1}{2}}d\sigma$.
\end{enumerate}
\end{lemma}

\begin{proof}
Let
\[
\e=\frac{1-\beta}{2}\int_{\Sp^{2n-1}}\varphi^{\frac{1}{2}}d\sigma,
\]
where $\beta$ is as in Lemma \ref{metricbinner2}. Let $R$ be the interior of a fundamental region in $\Sp^{2n-1}$ under the $G$-actions. Then $\sigma(\overline{R}\setminus R)=0$. To each $\eta\in R$ associate a family of metric balls $\{Q_{\delta}(\eta)\}_{\delta<\delta(\eta)}$ such that for each element in the family, $Q_{\delta}(\eta)\subset R$ and
\[
\sup_{z\in Q_{\delta}(\eta)}\varphi(z)-\inf_{z\in Q_{\delta}(\eta)}\varphi(z)<\frac{1}{2}\left(\frac{\e}{2|G|\sigma(\Sp^{2n-1})}\right)^2,
\]
where $|G|$ is the cardinality of the finite group $G$. This can be done since $R$ is a relatively open set in $\Sp^{2n-1}$, and since $\varphi$ is continuous on $\overline{R}$, so $\varphi$ is uniformly continuous on $R$. Moreover, the metric balls can be chosen as closed balls if we rescale the upper bound $\delta(\eta)$ for each $\eta\in R$. Then
\[
\cup_{\eta\in R}\{Q_{\delta}(\eta)\}_{\delta<\delta(\eta)}
\]
forms a Vitali cover of $R$. Since the metric balls with $\sigma$ satisfy the englfing and doubling property, an application of Vitali's covering theorem shows that there is a countable subcollection of disjoint metric balls $\{Q_j\}$ such that
\[
\sigma(R\setminus\cup_jQ_j)=0.
\]
(See \cite[Chap. 4 \S3]{S64} for the detail of Vitali's covering theorem and \cite[Chap. 1 \S5.4]{Ste70}.) 
By continuity, $|\varphi|\le M$ on $\Sp^{2n-1}$ for some $M>0$. Since $\{Q_j\}$ is a disjoint subcollection and since $\sigma(R)<\infty$, there is a finite set $E$ of positive integers such that
\[
\sigma(R\setminus\cup_{j\in E}Q_j)<\frac{\e}{2M^{1/2}|G|}.
\]
Since $\varphi$ is positive and continuous, for each $j\in E$, we can associate a constant $c_j>0$ such that
\[
\frac{1}{2}\inf_{z\in Q_j}\varphi(z)<c_j<\inf_{z\in Q_j}\varphi(z)
\]
and
\[
0<\inf_{z\in Q_j}\varphi(z)-c_j<\frac{1}{2}\left(\frac{\e}{2|G|\sigma(\Sp^{2n-1})}\right)^2.
\]
Let $\chi_j$ be the characteristic function of $Q_j$. Since $E$ is finite and $\varphi$ is positive,
\[
\tau:=\inf_{\zeta\in R}\left(\varphi(\zeta)-\sum_{j\in E}c_j\chi_j(\zeta)\right)>0
\]
and
\[
\tau\le c_j
\]
for all $j\in E$. Moreover,
\begin{align*}
\int_{R}\left(\varphi-\sum_{j\in E}c_j\chi_j\right)^{1/2}d\sigma
&=\int_{R\setminus\cup_{j\in E}Q_j}\varphi^{1/2}d\sigma+\sum_{j\in E}\int_{Q_j}\left(\varphi-c_j\right)^{1/2}d\sigma\\
&<\frac{\e}{2M^{1/2}|G|}\cdot M^{1/2}+\sum_{j\in E}\int_{Q_j}\left(\sup_{Q_j}\varphi-\inf_{Q_j}\varphi+\inf_{Q_j}\varphi-c_j\right)^{1/2}d\sigma\\
&<\frac{\e}{2|G|}+\sum_{j\in E}\int_{Q_j}\frac{\e}{2|G|\sigma(\Sp^{2n-1})}d\sigma\\
&\le\frac{\e}{2|G|}+\frac{\e}{2|G|\sigma(\Sp^{2n-1})}\cdot\sigma(\Sp^{2n-1})\\
&=\frac{\e}{|G|}.
\end{align*}

Note that $\kappa R$ is another fundamental region for any $\kappa\in G$. Then the metric balls $\kappa Q_j$ cover the new region $\kappa R$ with the same $E$ and constants. Therefore,
\[
\{Q_k\}:=\cup_{\kappa\in G,j\in E}\{\kappa Q_j\}
\]
forms a finite sub-collection of the cover $\cup_{\kappa\in G}\{\kappa Q_j\}$ for $\Sp^{2n-1}$ such that
\[
\inf_{\zeta\in\Sp^{2n-1}}\left(\varphi(\zeta)-\sum_{k}c_k\chi_k(\zeta)\right)=\tau
\]
and
\[
\int_{\Sp^{2n-1}}\left(\varphi-\sum_{k}c_k\chi_k\right)^{\frac{1}{2}}d\sigma<\e.
\]

For each $j\in E$, associate an $f_j$ to $Q_j$ as in Lemma \ref{metricbinner2} with $\frac{\tau}{2c_j|E|\cdot|G|}<1$ in place of $\tau$. Define
\[
F=\sum_{\kappa\in G}\sum_{j\in E}c_jf_j\circ\kappa.
\]
Note that for each $\kappa\in G$ and $j\in E$, the associated function to $\kappa Q_j=\kappa Q(\eta_j)=Q(\kappa(\eta_j))$ is simply $f_j\circ\kappa$. This is because the associated function in Lemma \ref{metricbinner} depends on $\langle z,\eta\rangle=\langle\kappa(z),\kappa(\eta)\rangle$ and $\kappa\in U(n)$. So the sum above can be written as the associated $c_kf_k$ over the finite sub-collection $\{Q_k\}$:
\[
F=\sum_{\kappa\in G}\sum_{j\in E}c_jf_j\circ\kappa=\sum_kc_kf_k.
\]

To verify (1), one sees each $f_k$ is holomorphic, continuous to $\B^n$, and $f(0)=0$. Moreover, $F$ is $G$-invariant, since by definition of $F$ the sum is over $G$.

To verify (2),
\[
\varphi-\Re(F)\ge\tau+\sum_kc_k(\chi_k-\Re(f_k)).
\]
The summand is $\ge0$ on $Q_k$ and is $>-\frac{\tau}{2|E|\cdot|G|}$ on $\Sp^{2n-1}\setminus Q_k$ by Lemma \ref{metricbinner2} (1) and (2). Hence,
\[
\varphi-\Re(F)\ge\tau-\frac{\tau}{2}>0.
\]

To verify (3), by Lemma \ref{metricbinner2} (3),
\begin{align*}
\int_{\Sp^{2n-1}}|F|^{\frac{1}{2}}d\sigma
&\le \sum_kc_k^{\frac{1}{2}}\int_{\Sp^{2n-1}}|f_k|^{\frac{1}{2}}d\sigma\\
&<\sum_kc_k^{\frac{1}{2}}\sigma(Q_k)=\int_{\Sp^{2n-1}}\left(\sum_kc_k\chi_k\right)^{\frac{1}{2}}d\sigma\\
&<\int_{\Sp^{2n-1}}\varphi^{\frac{1}{2}}d\sigma.
\end{align*}

To verify (4), by the choice of $\e$, the construction of the sub-collection $\{Q_k\}$, the previous step, and Lemma \ref{metricbinner2} (4),
\begin{align*}
\int_{\Sp^{2n-1}}(\varphi-\Re(F))^{\frac{1}{2}}d\sigma
&<\int_{\Sp^{2n-1}}\left(\varphi-\sum_kc_k\chi_k\right)^{\frac{1}{2}}d\sigma+\sum_kc_k^{\frac{1}{2}}\int_{\Sp^{2n-1}}|\chi_k-\Re(f_k)|^{\frac{1}{2}}d\sigma\\
&<\e+\beta\sum_kc_k^{\frac{1}{2}}\sigma(Q_k)\\
&<\e+\beta\int_{\Sp^{2n-1}}\varphi^{\frac{1}{2}}d\sigma=\frac{1+\beta}{2}\int_{\Sp^{2n-1}}\varphi^{\frac{1}{2}}d\sigma.
\end{align*}
The statement is true with $\gamma=\frac{1+\beta}{2}<1$, since $\beta<1$.
\end{proof}

\begin{proof}[Proof of Theorem \ref{normsquareb}]
For a positive $G$-invariant $\varphi\in C(\Sp^{2n-1})$, successive applications of Lemma \ref{mainlemball}, one obtains a sequence of holomorphic functions $\{F_k\}_{k=1}^{\infty}$:
\begin{enumerate}
\item $F_k\in C(\overline{\B^n})$, $F_k(0)=0$, and $F_k$ is $G$-invariant,
\item $\varphi-\Re(F_1+\cdots+F_k)>0$ on $\Sp^{2n-1}$,
\item $\int_{\Sp^{2n-1}}|F_k|^{\frac{1}{2}}d\sigma<\gamma^{k-1}\int_{\Sp^{2n-1}}\varphi^{\frac{1}{2}}d\sigma$,
\item $\int_{\Sp^{2n-1}}(\varphi-\Re(F_1+\cdots+F_k))^{\frac{1}{2}}d\sigma<\gamma^k\int_{\Sp^{2n-1}}\varphi^{\frac{1}{2}}d\sigma$,
\end{enumerate}
where $\gamma < 1$ is a constant.
Define
\[
h=\sum_{k=1}^{\infty}F_k,
\]
and the convergence holds uniformly on compact subsets in $\B^n$ as well as in  $H^{\frac{1}{2}}(\B^n)$. Moreover, $h$ is $G$-invariant and $\Re(h^*)=\varphi$ almost everywhere on $\Sp^{2n-1}$.

For the $G$-invariant function $\frac{1}{2}\log w\in L^1(\Sp^{2n-1})$, it can be written as
\[
\frac{1}{2}\log w=C-\lim_{j\to\infty}\varphi_j
\]
for some constant $C$ and some monotone positive $G$-invariant sequence $\varphi_j\in C(\Sp^{2n-1})$. Apply the previous argument to each $\varphi_j$, one obtains $G$-invariant $h_j\in H^{\frac{1}{2}}(\B^n)$. Then
\[
h:=C-\lim_{j\to\infty} h_j
\]
is the $G$-invariant holomorphic function on $\B^n$ with $\Re(h^*)=\frac{1}{2}\log w$ almost everywhere on $\Sp^{2n-1}$. To finish the proof, one finds $g=e^h$ as desired.
\end{proof}

\begin{remark}
In particular, when $w$ is strictly positive and continuous on $\Sp^{2n-1}$, the construction above guarantees $g\in H^{\infty}(\B^n)$. Therefore, the Hardy spaces $H^p(\Omega)$ associated to this $g$ are compatible with those defined in the classical sense, when the quotient domain $\Omega$ is sufficiently smooth. 
\end{remark}

\subsection{The Szeg\H{o} projection on generalized Thullen domains}

In this subsection, we will apply Theorem \ref{normsquareb} and Theorem \ref{main1} to study the Szeg\H{o} projection on generalized Thullen domains in $\C^2$.

\begin{definition}
For $m,k\in\Z^+$, the domain
\[
\Omega_{m,k}=\{(\zeta_1,\zeta_2)\in\C^2: |\zeta_1|^{\frac{2}{m}}+|\zeta_2|^{\frac{2}{k}}<1\}
\]
is called a generalized Thullen domain of type $(m,k)$ in $\C^2$. In particular, when $m=k=2$,
\[
\Omega_{2,2}=\{(\zeta_1,\zeta_2)\in\C^2: |\zeta_1|+|\zeta_2|<1\}
\]
is the classical Reinhardt triangle.
\end{definition}

Take $X=\B^2$ and $\Phi(z)=(z_1^m,z_2^k)$. The covering map $\Phi:\B^2\to\Omega_{m,k}$ extends smoothly to the boundary $\Sp^3$. The group $G$ is a finite subgroup of $U(2)$ and it preserves $\Sp^3$.

Using spherical coordinates
\[
\left\{\begin{array}{l} z_1=\cos\varphi_1+i\sin\varphi_1\cos\varphi_2 \\ z_2=\sin\varphi_1\sin\varphi_2\cos\varphi_3+i\sin\varphi_1\sin\varphi_2\sin\varphi_3 \end{array}\right.
\]
where $\varphi_3\in[0,2\pi]$ and $\varphi_1,\varphi_2\in[0,\pi]$, one obtains
\[
J_{\R}Z=\left(\begin{array}{cc} -\sin\varphi_1+i\cos\varphi_1\cos\varphi_2 & \cos\varphi_1\sin\varphi_2\cos\varphi_3+i\cos\varphi_1\sin\varphi_2\sin\varphi_3 \\ -i\sin\varphi_1\sin\varphi_2 & \sin\varphi_1\cos\varphi_2\cos\varphi_3+i\sin\varphi_1\cos\varphi_2\sin\varphi_3 \\ 0 & -\sin\varphi_1\sin\varphi_2\sin\varphi_3+i\sin\varphi_1\sin\varphi_2\cos\varphi_3\end{array}\right).
\]
Since
\[
J_{\C}\Phi=\left(\begin{array}{cc} mz_1^{m-1} \\ & kz^{k-1}_2 \end{array}\right),
\]
a direct computation shows
\[
A=\left(\begin{array}{cc} mz_1^{m-1}( -\sin\varphi_1+i\cos\varphi_1\cos\varphi_2) & kz_2^{k-1}\cos\varphi_1\sin\varphi_2(\cos\varphi_3+i\sin\varphi_3) \\ -imz_1^{m-1}\sin\varphi_1\sin\varphi_2 & kz_2^{k-1}\sin\varphi_1\cos\varphi_2(\cos\varphi_3+i\sin\varphi_3) \\ 0 & kz_2^{k-1}\sin\varphi_1\sin\varphi_2(-\sin\varphi_3+i\cos\varphi_3) \end{array}\right)
\]
and
\[
\Re(A\overline{A^t})=\left(\begin{array}{ccc} a_{11} & a_{12} & 0 \\ a_{21} & a_{22} & 0 \\ 0 & 0 & k^2|z_2|^{2k-2}\sin^2\varphi_1\sin^2\varphi_2 \end{array}\right),
\]
where
\[
a_{11}=m^2|z_1|^{2m-2}(\sin^2\varphi_1+\cos^2\varphi_1\cos^2\varphi_2)+k^2|z_2|^{2k-2}\cos^2\varphi_1\sin^2\varphi_2,
\]
\[
a_{12}=a_{21}=\sin\varphi_1\cos\varphi_1\sin\varphi_2\cos\varphi_2(-m^2|z_1|^{2m-2}+k^2|z_2|^{2k-2}),
\]
and
\[
a_{22}=m^2|z_1|^{2m-2}\sin^2\varphi_1\sin^2\varphi_2+k^2|z_2|^{2k-2}\sin^2\varphi_1\cos^2\varphi_2.
\]
Using $|z_1|^2= \cos^2\varphi_1+\sin^2\varphi_1\cos^2\varphi_2, |z_2|^2=\sin^2\varphi_1\sin^2\varphi_2$,  the direct computation shows
\[
\det[\Re(A\overline{A^t})]=m^2k^2|z_1|^{2m-2}|z_2|^{2k}\sin^2\varphi_1\left(m^2|z_1|^{2m-2}|z_2|^2+k^2|z_1|^2|z_2|^{2k-2}\right).
\]
On the other hand, the surface measure on $\Sp^3$, using spherical coordinates, is
\[
d\sigma(z)=\sin^2\varphi_1\sin\varphi_2d\varphi_1d\varphi_2d\varphi_3=|z_2|\sin\varphi_1d\varphi_1d\varphi_2d\varphi_3.
\]
Therefore, by Proposition \ref{den}, the density function $w$ in  on $\Sp^3$ is
\begin{equation}\label{reinJ}
w(z)=mk|z_1|^{m-1}|z_2|^{k-1}\left(m^2|z_1|^{2m-2}|z_2|^2+k^2|z_1|^2|z_2|^{2k-2}\right)^{\frac{1}{2}}.
\end{equation}

The following formula is useful when computing integrals over a sphere.

\begin{theorem}[Forelli's Formula]
\label{Forelli}
Suppose $1\le l<n$ and $f$ is a function on $\Sp^{2n-1}$ that depends only on $z_1,\dots,z_{l}$. Then $f$ can be regarded as defined on $\overline{\B^{l}}$. Letting $\pi:\C^{n}\to\C^{l}$ be the orthogonal projection, then
\[
\int_{\Sp^{2n-1}}f\circ\pi d\sigma=\left(\begin{array}{c} n-1 \\ l \end{array}\right)\int_{\B^{l}}(1-|\zeta|^2)^{n-l-1}f(\zeta)dV_{l}(\zeta)
\]
for every $f\in L^1(\B^{l})$.
\end{theorem}

\begin{proof}
See \cite[1.4.4(1)]{R80} and \cite[p. 381]{For74}.
\end{proof}

\begin{proposition}
\label{reintriallow}
The density function $w$ on $\Sp^3$ given by \eqref{reinJ} is admissible.
\end{proposition}

\begin{proof}
By continuity, $\log w$ is integrable when $w$ is away from $0$. Since $w$ is approaching to $0$ either as $|z_1|\to0^+$ or $|z_2|\to0^+$, we only need to check these two cases. Note that
\begin{equation}\label{sizem}
w(z)\approx\left\{\begin{array}{lc} |z_1|^m & \text{when}\,\, m\ge2 \\ 1 & \text{when}\,\, m=1 \end{array}\right.
\end{equation}
as $|z_1|\to0^+$, and
\begin{equation}\label{sizek}
w(z)\approx\left\{\begin{array}{lc} |z_2|^k & \text{when}\,\, k\ge2 \\ 1 & \text{when}\,\, k=1 \end{array}\right.
\end{equation}
as $|z_2|\to0^+$. By Theorem \ref{Forelli}, one sees
\[
\int_{\Sp^3}\chi_{\{|z_1|<\e\}}\cdot |\log w| d\sigma\approx\int_{\{|z_1|<\e\}}|\log|z_1|^m|dA(z_1)\approx\int_0^{\e}-r_1\log r_1dr_1<\infty
\]
when $m\ge2$, and
\[
\int_{\Sp^3}\chi_{\{|z_1|<\e\}}\cdot |\log w| d\sigma\le\int_{\{|z_1|<\e\}}C_1\,\,dA(z_1)<\infty
\]
for some constant $C_1>0$ when $m=1$. Similarly,
\[
\int_{\Sp^3}\chi_{\{|z_2|<\e\}}\cdot |\log w| d\sigma\approx\int_{\{|z_2|<\e\}}|\log|z_2|^k|dA(z_2)\approx\int_0^{\e}-r_2\log r_2dr_2<\infty
\]
when $k\ge2$, and
\[
\int_{\Sp^3}\chi_{\{|z_2|<\e\}}\cdot |\log w| d\sigma\le\int_{\{|z_2|<\e\}}C_2\,\,dA(z_2)<\infty
\]
for some constant $C_2>0$ when $k=1$. So $\log w\in L^1(\Sp^3)$, and hence $w$ is admissible by Theorem \ref{normsquareb}. 
\end{proof}

Recall the general theory on strongly pseudoconvex domains with $C^2$ boundary in \cite{WW21}. We then apply it to the unit sphere, and we also call the density function with respect to the surface measure $d\sigma$ the weight function.

\begin{definition}
\label{Apsphere}
For $1<p<\infty$, a weight $\mu$ on $\Sp^{2n-1}$ is in $A_p(\Sp^{2n-1})$ if
\[
\sup_Q[\mu]_{p,Q}:=\sup_{Q}\left(\frac{1}{\sigma(Q)}\int_Q\mu d\sigma\right)\left(\frac{1}{\sigma(Q)}\int_Q\mu^{\frac{-1}{p-1}}d\sigma\right)^{p-1}<\infty,
\]
where $Q$ denotes the metric balls in Definition \ref{metricball}.
\end{definition}

\begin{theorem}[Theorem 1.1 for $\B^n$ in \cite{WW21}]
\label{weightLpsphere}
For $1<p<\infty$,  if $\mu\in A_p(\Sp^{2n-1})$, the Szeg\H{o} projection $S_{\B^n}$ on $\B^n$ is bounded on $L^p(\Sp^{2n-1}, \mu)$.
\end{theorem}

The following lemma is the key for the verification of $A_p$ condition.

\begin{lemma}
\label{intsize}
Let $E(\delta)=\{\lambda\in\D:\,\,|1-\lambda|<\delta^2\}$ for $0<\delta\le\sqrt{2}$ and
\[
I(\alpha,\delta)=\int_{E(\delta)}(1-|\lambda|^2)^{\alpha}dA(\lambda).
\]
If $\alpha\le-1$, the integral $I(\alpha,\delta)$ diverges; if $\alpha>-1$, the integral $I(\alpha,\delta)$ converges and
\[
\lim_{\delta\to0^+}\frac{I(\alpha,\delta)}{\delta^{2\alpha+4}}=\gamma_{\alpha}
\]
for some constant $\gamma_{\alpha}>0$. Moreover,  when $\alpha\ge0$, the ratio above increases to $\gamma_{\alpha}$ as $\delta$ decreases to $0$.
\end{lemma}

\begin{proof}
For $\alpha\le-1$,
by elementary geometry, one sees
\[
R(\e,\eta)=\{\lambda=re^{i\theta}:\,\,1-\e<r<1\,\,\text{and}\,\,-\eta<\theta<\eta\}\subset E(\delta)
\]
for some $\e=\e(\delta)>0$ and $\eta=\eta(\delta)>0$ sufficiently small. Then
\begin{align*}
I(\alpha,\delta)
&\ge\int_{R(\e,\eta)}(1-|\lambda|^2)^{\alpha}dA(\lambda)\\
&=2\eta\int_{1-\e}^1(1-r^2)^{\alpha}rdr=\infty.
\end{align*}

For $\alpha>-1$, the change of variables $1-\lambda=\frac{\delta^2}{z}$ gives
\begin{align*}
I(\alpha,\delta)
&=\int_{E'(\delta)}\left(1-\left|1-\frac{\delta^2}{z}\right|^2\right)^{\alpha}\left|\frac{\delta^2}{z^2}\right|^2dA(z)\\
&=\int_{E'(\delta)}\frac{\delta^4\left(|z|^2-|z-\delta^2|^2\right)^{\alpha}}{|z|^{2\alpha+4}}dA(z)\\
&=\delta^{2\alpha+4}\int_{E'(\delta)}\frac{\left(2\Re(z)-\delta^2\right)^{\alpha}}{|z|^{2\alpha+4}}dA(z)=:\delta^{2\alpha+4}J(\alpha,\delta),
\end{align*}
where $E'(\delta)=\{|z|>1:\,\,2\Re(z)>\delta^2\}$. When $\alpha\ge0$, as $\delta$ decreases to $0$, the integrand and the domain of integration in $J(\alpha,\delta)$ both increase. The monotone convergence theorem shows
\begin{equation}
\label{deltalimit}
\begin{split}
\lim_{\delta\to0^+}\frac{I(\alpha,\delta)}{\delta^{2\alpha+4}}=\lim_{\delta\to0^+}J(\alpha,\delta)
&=\int_{E'(0)}\frac{\left(2\Re(z)\right)^{\alpha}}{|z|^{2\alpha+4}}dA(z)\\
&=2^{\alpha}\int_{1}^{\infty}\frac{dr}{r^{\alpha+3}}\int_{-\pi/2}^{\pi/2}\cos^{\alpha}\theta d\theta\\
&=\frac{2^{\alpha+1}}{\alpha+2}\cdot B\left(\frac{1}{2},\frac{\alpha+1}{2}\right)=:\gamma_{\alpha}>0,
\end{split}
\end{equation}
where $B(t_1,t_2)=\int_0^{\pi/2}(\sin\theta)^{2t_1-1}(\cos\theta)^{2t_2-1} d\theta$ for $t_1,t_2>0$ is the beta function.

When $-1<\alpha<0$, another change of variables
\[
\left\{\begin{array}{l} 2x=2\Re(z)-\delta^2 \\ y=\Im(z) \end{array}\right.
\]
in $J(\alpha,\delta)$ gives
\[
J(\alpha,\delta)=\int_{E''(\delta)}\frac{2^{\alpha}x^{\alpha}}{\left[\left(x+\frac{\delta^2}{2}\right)^2+y^2\right]^{\alpha+2}}dxdy,
\]
where $E''(\delta)=\left\{(x,y)\in\R^2:\,\,x>0\,\,\text{and}\,\,\left(x+\frac{\delta^2}{2}\right)^2+y^2>1\right\}$. Note that for $\delta=\sqrt{2}$
\[
I(\alpha,\sqrt{2})=\int_{\D}\left(1-|\lambda|^2\right)^{\alpha}dA(\lambda)<\infty,
\]
since $\alpha>-1$. For $0<\delta<\sqrt{2}$, by elementary geometry, there is an $\epsilon=\epsilon(\delta)>0$ such that
\[
\left\{(x,y):\,\,x^2+y^2\le\epsilon^2\right\}\subset\left\{(x,y):\,\,\left(x+\delta^2/2\right)^2+y^2\le1  \right\}.
\]
So
\begin{equation}
\label{dominate}
\begin{split}
J(\alpha,\delta)
&\le\int_{\{x>0,x^2+y^2>\epsilon^2\}}\frac{2^{\alpha}x^{\alpha}}{\left(x^2+y^2\right)^{\alpha+2}}dxdy\\
&=2^{\alpha}\int_{\epsilon}^{\infty}\frac{dr}{r^{\alpha+3}}\int_{-\pi/2}^{\pi/2}\cos^{\alpha}\theta d\theta\\
&=\frac{2^{\alpha+1}}{(\alpha+2)\epsilon^{\alpha+2}}\cdot B\left(\frac{1}{2},\frac{\alpha+1}{2}\right)<\infty,
\end{split}
\end{equation}
which implies $I(\alpha,\delta)<\infty$. Furthermore, for $0<\delta<1$ and $\delta\to0^+$, $\epsilon=\epsilon(1)$ can be chosen uniformly.  The estimate in \eqref{dominate} and  the dominated convergence theorem guarantees that the limit still holds in \eqref{deltalimit}.
\end{proof}

\begin{remark}
This lemma generalizes \cite[Proposition 5.1.4]{R80} where a special case $\alpha=n-2\ge0$ is treated.
\end{remark}

\begin{proposition}
\label{doublemeasure}
There are constants $c_1>1$ and $c_2>1$ such that for all $\eta_1,\eta_2\in\Sp^3$ and $0<\delta\le\sqrt{2}$, the metric balls $Q$ in $\Sp^3$ satisfy the following properties:
\begin{enumerate}
\item if $Q_{\delta}(\eta_1)\cap Q_{\delta}(\eta_2)\neq\emptyset$, then $Q_{\delta}(\eta_2)\subset Q_{c_1\delta}(\eta_1)$;
\item $\sigma(Q_{\delta}(\eta_1))\approx\delta^4$ and indeed $\sigma(Q_{c_1\delta}(\eta_1))\le c_2\sigma(Q_{\delta}(\eta_1))$.
\end{enumerate}
\end{proposition}

\begin{proof}
When $Q_{\delta}(\eta_1)\cap Q_{\delta}(\eta_2)\neq\emptyset$, since $d$ is a metric, the triangle inequality shows that for any $\zeta\in Q_{\delta}(\eta_2)$
\[
d(\zeta,\eta_1)<3\delta,
\]
and thus $\zeta\in Q_{3\delta}(\eta_1)$. So the first property holds for $c_1=3$.

Since $d$ is invariant under $U(n)$, by rotational symmetry, it suffices to check $\eta_1=(1,0)$ in the second property. By Definition \ref{metricball}
\[
Q_{\delta}(1,0)=\{|z_1|^2+|z_2|^2=1,\,|z_1-1|<\delta^2\}.
\]
By Theorem \ref{Forelli} and Lemma \ref{intsize},
\[
\sigma(Q_{\delta}(1,0))=\int_{\{|z_1-1|<\delta^2\} \cap \{|z_1|<1\}}dA(z_1)=I(0,\delta)\approx\delta^4
\]
for any $0<\delta\le\sqrt{2}$. Moreover, since $\delta<c_1\delta$, by Lemma \ref{intsize},
\[
\frac{\sigma(Q_{c_1\delta}(1,0))}{(c_1\delta)^4}=\frac{I(0,c_1\delta)}{(c_1\delta)^4}\le\frac{I(0,\delta)}{\delta^4}=\frac{\sigma(Q_{\delta}(1,0))}{\delta^4},
\]
which proves the second claim by taking $c_2=c_1^4$.
\end{proof}

\begin{remark}
Statement (1) is the engulfing property and statement (2) is the doubling property.
\end{remark}

We are now able to derive the $L^p$ regularity of the Szeg\H{o} projection on $\Omega_{m,k}$.

\begin{theorem}\label{LpReinTri}
For $m\ge2$ or $k\ge2$, the Szeg\H{o} projection on $\Omega_{m,k}$ is $L^p$ bounded if and only if $p\in\left(\frac{2m+4}{m+4},\frac{2m+4}{m}\right)\cap\left(\frac{2k+4}{k+4},\frac{2k+4}{k}\right)$.
For $m=k=1$, the Szeg\H{o} projection is $L^p$ bounded if $p\in(1,\infty)$.
\end{theorem}

\begin{proof}
By Theorem \ref{main1}, the Szeg\H{o} projection $S_{\Omega_{m,k}}$ is bounded on $L^p(\partial\Omega_{m,k})$ if and only if $S_{\B^2}$ is bounded on $G$-invariant functions in $L^p(\Sp^3,w^{1-\frac{p}{2}})$.

It suffices to prove the case $m\ge2$ or $k\ge2$, since the case $m=k=1$ corresponds to the ball. 
To obtain the boundedness, by Theorem \ref{weightLpsphere}, it suffices to check $w^{1-\frac{p}{2}}\in A_p(\Sp^3)$. Note that the zeros of $w(z)$ on $\Sp^3$ are the two circles 
\[
Z_1=\{|z_1|=1,\,\,\,z_2=0\}\qquad\text{and}\qquad Z_2=\{|z_2|=1,\,\,\,z_1=0\}.
\]
Define
\[
L:=\inf\{d(z,\zeta):\,\,z\in Z_1,\zeta\in Z_2\}>0
\]
to be the distance between $Z_1$ and $Z_2$, where $d$ is the metric in Definition \ref{metricball}.
Let $Q=Q_{\delta}(\eta)$, and
let
\[
L_1:=\inf\{d(\eta,\zeta):\,\,\zeta\in Z_1\}\qquad\text{and}\qquad L_2:=\inf\{d(\eta,\zeta):\,\,\zeta\in Z_2\}
\]
be the distance from $\eta$ to $Z_1$ and $Z_2$, respectively. Note $L_1\leq \sqrt{2}, L_2 \le\sqrt{2}$, since any distance on $\Sp^3$ is not greater than $\sqrt{2}$ by the definition of the metric $d$.  We divide the argument into three cases.

Case one: $L_1\ge2\delta$ and $L_2\ge2\delta$. By triangle inequality, for any $z\in Q$ the distance from $z$ to $Z_1$ is at least $L_1-\delta\ge \frac{L_1}{2}$. That is, $d(z,\zeta)\ge \frac{L_1}{2}$ for all $\zeta\in Z_1$, which implies
\[
|e^{i\theta}-z_1|\ge \left(\frac{L_1}{2}\right)^2
\]
for all $\theta\in[0,2\pi]$. So $|z_1|\le1-\frac{L_1^2}{4}$, and hence $|z_2|>\frac{L_1}{3}$. Without loss of generality, assume $d(\eta,(1,0))=L_1$, then for any $z\in Q$ the distance from $z$ to $(1,0)$ is at most $L_1+\delta\le \frac{3L_1}{2}$. So
\[
|1-z_1|\le\left(\frac{3L_1}{2}\right)^2,
\]
which implies $|z_1|>1-\frac{9}{4} L_1^2$. Hence, $|z_2|<3L_1$. Similar argument shows
\[
\frac{1}{3}L_2<|z_1|<3L_2
\]
by the relation between $z$ and $Z_2$. By the estimates \eqref{sizem} and \eqref{sizek}, $w^{1-\frac{p}{2}}\approx L_1^{k(1-\frac{p}{2})}$ as $|z_2|\to0^+$ and $w^{1-\frac{p}{2}}\approx L_2^{m(1-\frac{p}{2})}$ as $|z_1|\to0^+$ (or bounded above and below when $k=1$ or $m=1$, respectively). As a result $[w^{1-\frac{p}{2}}]_{p,Q}<\infty$.

Case two: $L_1<2\delta$ or $L_2<2\delta$, but $\delta\ge\delta_0$ for some small fixed constant $0<\delta_0<<L$. By Proposition \ref{doublemeasure}, $\sigma(Q)\ge C_0\delta_0^4$ for some $C_0>0$. Since $Q\subset\Sp^3$,
\[
\int_Qw^{1-\frac{p}{2}}d\sigma\le\int_{\Sp^3}w^{1-\frac{p}{2}}d\sigma.
\]
As in Proposition \ref{reintriallow}, $w^{1-\frac{p}{2}}$ is integrable away from its zeros. By \eqref{sizem} and \eqref{sizek} and Theorem \ref{Forelli},
\[
\int_{\Sp^3}\chi_{\{|z_1|<\e\}}\cdot w^{1-\frac{p}{2}}d\sigma\approx\int_{\{|z_1|<\e\}}|z_1|^{m(1-\frac{p}{2})}dA(z_1)= C_1(m,p,\e)<\infty
\]
provided $m(1-\frac{p}{2})+2>0$ when $m\ge2$ (when $m=1$ it is trivial and no restriction on $p$); and
\[
\int_{\Sp^3}\chi_{\{|z_2|<\e\}}\cdot w^{1-\frac{p}{2}}d\sigma\approx\int_{\{|z_2|<\e\}}|z_2|^{k(1-\frac{p}{2})}dA(z_2)= C_1(k,p,\e)<\infty
\]
provided $k(1-\frac{p}{2})+2>0$ when $k\ge2$ (when $k=1$ it is trivial and no restriction on $p$). A similar argument applied to $w^{(1-\frac{p}{2})\frac{-1}{p-1}}$ shows
\[
\int_Qw^{(1-\frac{p}{2})\frac{-1}{p-1}}d\sigma\le\int_{\Sp^3}w^{(1-\frac{p}{2})\frac{-1}{p-1}}d\sigma\le C_2(m,k,p,\e)<\infty
\]
provided $m(1-\frac{p}{2})\frac{-1}{p-1}+2>0$ when $m\ge2$ (when $m=1$ no restriction on $p$) and $k(1-\frac{p}{2})\frac{-1}{p-1}+2>0$ when $k\ge2$ (when $k=1$ no restriction on $p$). Solving the inequalities, we have
\[
[w^{1-\frac{p}{2}}]_{p,Q}\le C(m,k,p,\delta_0)<\infty
\]
for $p\in\left(\frac{2m+4}{m+4},\frac{2m+4}{m}\right)\cap\left(\frac{2k+4}{k+4},\frac{2k+4}{k}\right)$ when $m\ge2$ or $k\ge2$.

Case three: $L_1<2\delta$ or $L_2<2\delta$ and $\delta<\delta_0<< L$. For $L_1<2\delta$, $L_2>>\delta_0>\delta$, the only zeros of $w$ will be in $Z_1$. Without loss of generality, assume $d(\eta,(1,0))=L_1<2\delta$,  then
\[
Q_{\delta}(1,0)\cap Q_{\delta}(\eta)\neq\emptyset.
\]
By Proposition \ref{doublemeasure}, $Q\subset Q_{c_1\delta}(1,0)$ for some $c_1>1$ and $\sigma(Q)\approx\delta^4$. When $k\ge2$, by the estimate in \eqref{sizek}, Theorem \ref{Forelli}, and Lemma \ref{intsize} as $\delta<\delta_0$ sufficiently small,

\begin{align*}
\frac{1}{\sigma(Q)}\int_{Q}w^{1-\frac{p}{2}} d\sigma
&\le\frac{1}{\sigma(Q)}\int_{Q_{c_1\delta}(1,0)}|z_2|^{k(1-\frac{p}{2})} d\sigma\\
&\approx\frac{1}{\delta^4}\int_{|1-z_1|^{\frac{1}{2}}<c_1\delta}|z_2|^{k(1-\frac{p}{2})}d\sigma\\
&=\frac{1}{\delta^4}\int_{\{|1-z_1|<(c_1\delta)^2\} \cap \{|z_1|<1\}}(1-|z_1|^2)^{\frac{k}{2}(1-\frac{p}{2})}dA(z_1)\\
&=\frac{1}{\delta^4}\cdot I\left(\frac{k}{2}(1-\frac{p}{2}),c_1\delta\right)\\
&\approx\frac{1}{\delta^4}\cdot\delta^{k(1-\frac{p}{2})+4}=\delta^{k(1-\frac{p}{2})}
\end{align*}
provided $\frac{k}{2}(1-\frac{p}{2})>-1$ and

\begin{align*}
\left(\frac{1}{\sigma(Q)}\int_{Q}w^{(1-\frac{p}{2})\frac{-1}{p-1}} d\sigma\right)^{p-1}
&\le\left(\frac{1}{\sigma(Q)}\int_{Q_{c_1\delta}(1,0)}|z_2|^{-\frac{k}{p-1}(1-\frac{p}{2})}d\sigma\right)^{p-1}\\
&\approx\left(\frac{1}{\delta^4}\int_{|1-z_1|^{\frac{1}{2}}<c_1\delta}|z_2|^{-k(1-\frac{p}{2})/(p-1)}d\sigma\right)^{p-1}\\
&=\left(\frac{1}{\delta^4}\int_{\{|1-z_1|<(c_1\delta)^2\}\cap\{|z_1|<1\}}(1-|z_1|^2)^{\frac{k}{2-2p}(1-\frac{p}{2})}dA(z_1)\right)^{p-1}\\
&=\left(\frac{1}{\delta^4}\cdot I\left(\frac{k}{2-2p}(1-\frac{p}{2}),c_1\delta\right)\right)^{p-1}\\
&\approx\left(\frac{1}{\delta^4}\cdot\delta^{\frac{k}{1-p}(1-\frac{p}{2})+4}\right)^{p-1}=\delta^{-k(1-\frac{p}{2})}
\end{align*}
provided $\frac{k}{2-2p}(1-\frac{p}{2})>-1$. Solving the inequalities, we have
\[
[w^{1-\frac{p}{2}}]_{p,Q}\le C(k,p,\delta_0) <\infty
\]
for $p\in\left(\frac{2k+4}{k+4},\frac{2k+4}{k}\right)$, when $k\ge2$ (when $k=1$ no restriction on $p$). Similarly, for $L_2<2\delta$
\[
[w^{1-\frac{p}{2}}]_{p,Q}\le C(m,p,\delta_0) <\infty
\]
for $p\in\left(\frac{2m+4}{m+4},\frac{2m+4}{m}\right)$, when $m\ge2$ (when $m=1$ no restriction on $p$).  This verifies the boundedness.

To check the unboundedness, consider the $G$-invariant test function
\[
h(z_1,z_2)=|z_1^mz_2^k|
\]
on $\Sp^3$. By rotational symmetry, a direct computation shows
\[
S_{\B^2}(h)=C_{m,k}
\]
and by the estimate in \eqref{sizek} and Theorem \ref{Forelli}
\[
\int_{\Sp^3}w(z)^{1-\frac{1}{2}\cdot\frac{2k+4}{k}}d\sigma(z)>\int_{|z_2|<\e}|z_2|^{-2}dA(z_2)=\infty.
\]
On the other hand,
\[
\int_{\Sp^3}|h|^{\frac{2k+4}{k}}w(z)^{1-\frac{1}{2}\cdot\frac{2k+4}{k}}d\sigma(z)\le C\int_{\Sp^3}|z_1|^{a}|z_2|^{b}d\sigma(z)<\infty,
\]
for some $a, b\ge0$. This shows the unboundedness for $p=\frac{2k+4}{k}$. The case for $p=\frac{2m+4}{m}$ is similar. By duality, this completes the proof.

\end{proof}

\begin{corollary}
The Szeg\H{o} projection on the classical Reinhardt triangle $\Omega_{2,2}$ is $L^p$ bounded if and only if $\frac{4}{3}<p<4$.
\end{corollary}

\subsection{Minimal ball}

The result for the classical Reinhardt triangle $\Omega_{2,2}$ can be applied to the following domain.
\begin{definition}
The domain defined by
\[
\B_*=\{(\xi_1,\xi_2)\in\C^2: |\xi_1|^2+|\xi_2|^2+|\xi_1^2+\xi_2^2|<1\}
\]
is called the minimal ball in $\C^2$.
\end{definition}

Let
\[
F(\zeta_1,\zeta_2)=\left(\frac{\zeta_1+\zeta_2}{\sqrt{2}},\frac{i\zeta_1-i\zeta_2}{\sqrt{2}}\right)
\]
be a linear map on $\C^2$. Then $F$ indeed maps $\Omega_{2,2}$ biholomorphically onto $\B_*$. 

A direct computation shows
\[
J_{\C}F(\zeta)=\left(\begin{array}{cc} \frac{1}{\sqrt{2}} & \frac{i}{\sqrt{2}} \\ \frac{1}{\sqrt{2}} & \frac{-i}{\sqrt{2}} \end{array}\right).
\]
Note that
\[
\left(\begin{array}{cc} \frac{1}{\sqrt{2}} & \frac{i}{\sqrt{2}} \\ \frac{1}{\sqrt{2}} & \frac{-i}{\sqrt{2}} \end{array}\right)\overline{\left(\begin{array}{cc} \frac{1}{\sqrt{2}} & \frac{i}{\sqrt{2}} \\ \frac{1}{\sqrt{2}} & \frac{-i}{\sqrt{2}} \end{array}\right)^t}=\left(\begin{array}{cc} \frac{1}{\sqrt{2}} & \frac{i}{\sqrt{2}} \\ \frac{1}{\sqrt{2}} & \frac{-i}{\sqrt{2}} \end{array}\right)\left(\begin{array}{cc} \frac{1}{\sqrt{2}} & \frac{1}{\sqrt{2}} \\ \frac{-i}{\sqrt{2}} & \frac{i}{\sqrt{2}} \end{array}\right)=\left(\begin{array}{cc} 1  \\  & 1 \end{array}\right).
\]
By \eqref{matrixA} and Proposition \ref{den}, the pullback density function for $\B_*$ is exactly the same as the pullback density function for $\Omega_{2,2}$. Hence the $L^p$ regularity of $S_{\B_*}$ is the same as that of $S_{\Omega_{2,2}}$.

\begin{corollary}
The Szeg\H{o} projection on the minimal ball $\B_*$ is bounded on $L^p(\partial\B_*)$ if and only if $\frac{4}{3}<p<4$.
\end{corollary}




{\bf Acknowledgement} The authors would like to thank Brett Wick for helpful communications. They also thank Xiaojun Huang for pointing out a mistake in the earlier version of the paper.

\bigskip


\bibliographystyle{plain}

\end{document}